\documentclass{article}

\usepackage{arxiv}

\usepackage{amsmath, amsthm, amssymb, epsfig, sectsty, graphicx, float, url}
\usepackage{algorithm, algorithmicx, algpseudocode}
\usepackage{float}
 \usepackage{subcaption}
 \usepackage{caption}

\usepackage[utf8]{inputenc} % allow utf-8 input
\usepackage[T1]{fontenc}    % use 8-bit T1 fonts
\usepackage{hyperref}       % hyperlinks
\usepackage{url}            % simple URL typesetting
\usepackage{booktabs}       % professional-quality tables
\usepackage{amsfonts}       % blackboard math symbols
\usepackage{nicefrac}       % compact symbols for 1/2, etc.
\usepackage{microtype}      % microtypography
\usepackage{lipsum}
\usepackage{bm}

\numberwithin{equation}{section}

\usepackage[colorinlistoftodos,bordercolor=orange,backgroundcolor=orange!20,line
color=orange,textsize=scriptsize]{todonotes}
\catcode`\@=11
\catcode`\@=12

\theoremstyle{plain}
\newtheorem{theorem}{Theorem}[section]
\newtheorem{lemma}[theorem]{Lemma}

\newtheorem{proposition}[theorem]{Proposition}

\theoremstyle{definition}
\newtheorem{definition}[theorem]{Definition}
\newtheorem{remark}[theorem]{Remark}

\newcommand{\R}{\mathbb{R}}

\newcommand{\A}{\mathbb{A}}
\newcommand{\B}{\mathcal{B}}
\newcommand{\D}{\mathcal{D}}
\newcommand{\F}{\mathcal{F}}
\newcommand{\N}{\mathbb{N}}
\newcommand{\X}{\mathbb{S}}

\newcommand{\V}{\mathcal{V}}

\newcommand{\C}{\mathbb{C}}

\title{A First-order Approach to Accelerated Value Iteration}

\author{%
  Vineet Goyal \\
  IEOR Department, Columbia University\\
  \texttt{vgoyal@ieor.columbia.edu} \\
   \And
   Julien Grand-Clement \\
   ISOM Department, HEC Paris \\
   \texttt{grand-clement@hec.fr} \\
}

\begin{document}

\maketitle

\begin{abstract}
Markov decision processes (MDPs) are used to model stochastic systems in many applications. Several efficient algorithms to compute optimal policies have been studied in the literature, including value iteration (VI) and policy iteration. However, these do not scale well especially when the discount factor for the infinite horizon discounted reward, $\lambda$, gets close to one. In particular, the running time scales as $O \left( 1/(1-\lambda) \right)$ for these algorithms. 
{
In this paper, our goal is to design new algorithms that scale better than previous approaches when $\lambda$ approaches $1$.  Our main contribution is to present  a connection between VI and \textit{gradient descent} and adapt the ideas of \textit{acceleration} and \textit{momentum} in convex optimization to design faster algorithms for MDPs. 
We prove theoretical guarantees of faster convergence of our algorithms for the 
computation of the value function of a policy, where the running times of our algorithms scale as $O \left( 1/\sqrt{1-\lambda} \right)$ for \textit{reversible} MDP instances. The improvement is quite analogous to Nesterov's acceleration and momentum in convex optimization.  We also provide a lower bound on the convergence properties of any first-order algorithm for solving MDPs, presenting a family of MDPs instances for which no algorithm can converge faster than VI when the number of iterations is smaller than the number of states.
We introduce a Safe Accelerated Value Iteration (S-AVI), which alternates between accelerated updates and value iteration updates. Our algorithm S-AVI is worst-case optimal and retains the theoretical convergence properties of VI while exhibiting strong empirical performances, providing significant speedups compared to classical approaches (up to one order of magnitude in many cases) for a large test bed of MDP instances.}
\end{abstract}


%%%%%%%%%%%%%%%%%%%%%%%%%%%%%%%%%%%%%%%%%%%%%%%%%%%%%%%%%%%%%%%%%%%%%%

% Samples of sectioning (and labeling) in OPRE
% NOTE: (1) \section and \subsection do NOT end with a period
%       (2) \subsubsection and lower need end punctuation
%       (3) capitalization is as shown (title style).
%
%\section{Introduction.}\label{intro} %%1.
%\subsection{Duality and the Classical EOQ Problem.}\label{class-EOQ} %% 1.1.
%\subsection{Outline.}\label{outline1} %% 1.2.
%\subsubsection{Cyclic Schedules for the General Deterministic SMDP.}
%  \label{cyclic-schedules} %% 1.2.1
%\section{Problem Description.}\label{problemdescription} %% 2.

% Text of your paper here
\section{Introduction.}
Markov Decision Processes (MDPs) are widely used to model sequential decision-making problems under uncertainty. The goal is to find a policy that maximizes the infinite horizon discounted reward, for a (fixed) \textit{discount factor} $\lambda \in (0,1)$. Several algorithms have been studied in the literature, including policy iteration (PI,~\cite{howard-1960}) and or value iteration (VI); we refer the reader to \cite{Puterman} for an extensive review of MDPs. 
However, none of these algorithms scale well when $\lambda$ is close to $1$, which is the case in many applications where the effective horizon is large. In particular, the number of iterations of VI {and PI} before convergence scales as $O \left( 1/(1-\lambda) \right)$ {(see Section \ref{sec:MDP} for more detailed convergence rates).}
{Each iteration of VI and PI requires to estimate the Bellman operator at the current iterate,  which becomes computationally expensive for large state and action spaces, and requires to compute high-dimensional expectations. }
{Algorithms for the linear programming formulation of MDPs may scale better as regards  the discount factor but require more computation per iteration~\citep{ye-2011}.}

The goal in this paper is to design  algorithms to solve MDPs that are scalable and more efficient than current approaches.  { In particular,  our goal is to design algorithms that converge in a smaller number of iterations than VI, requiring less evaluations of the Bellman operator and resulting in significant empirical speedups.}
Several faster algorithms have been proposed in this direction in the literature. The most widespread are the \textit{Gauss-Seidel} and \textit{Jacobi}  value iteration algorithms, which iteratively apply operators that are variations of the Bellman operator (see Section 6.3.3 in~\cite{Puterman}). The authors in \cite{herzberg-1996} propose iterative algorithms based on one-step and $k$-step look-ahead, while \cite{shlakhter-2010} compose the Bellman operator with a projection. Even though these algorithms can be proved to converge at linear rate at least as fast as VI, the exact convergence rate is not known. 
\subsection{Our Contributions.}
Our main contributions are as follows.

\vspace{1mm}
\noindent \textbf{Connection between Gradient Descent and Value Iteration.} We present a fundamental analogy between gradient descent (GD) and value iteration (VI). In particular, for the Bellman operator $T$ in value iteration, we consider $\bm{v} - T(\bm{v})$ as the gradient of some function at $\bm{v} \in \R^{n}$. Therefore, the {\em gradient} of this function vanishes at the fixed-point of $T$.  Building upon this,  we propose an iterative algorithm, \textit{Relaxed Value Iteration} (\ref{alg:R-VI}).  { Our analogy between the properties of VI and GD is crucial to choose the step sizes in all the algorithms in our paper.} In particular, we prove convergence of \ref{alg:R-VI} at a linear rate that surprisingly matches the convergence and step sizes of gradient descent. %We would like to note that Bertsekas and Tsitsiklis~\cite{neuro-dp} suggest that one can regard the vector $\bm{v} - T(\bm{v})$ as the gradient of some function at $\bm{v} \in \R^{n}$

\vspace{1mm}
\noindent \textbf{Accelerated Value Iteration.} Our main contribution is to present an accelerated algorithm for value iteration.  Our algorithms, Accelerated Value Iteration \eqref{alg:AVI} and Momentum Value Iteration \eqref{alg:MVI}, rely upon the idea of Nesterov's acceleration and Polyak's Momentum in convex optimization and choosing the step sizes appropriately, {based on our analogy between value iteration and gradient descent}. While we borrow ideas from convex optimization, the proofs for convergence of gradient descent algorithms cannot be directly extended to our setting, since the Bellman operator is neither smooth nor convex. 

\vspace{1mm}
\noindent \textbf{Insights for Accelerated Value Computation.} We derive some insights in the performances of accelerated algorithms for MDPs for the special case of computing the value function of a policy (\textit{Value Computation} (VC)). In this case,  the operator is affine and we show strong theoretical convergence guarantees for Accelerated VC (\ref{alg:AVC}) and Momentum VC (\ref{alg:MVC}).  In particular, when the Markov chain induced by a policy $\pi$ is \textit{irreducible} and \textit{reversible}, Algorithm \ref{alg:AVC} and Algorithm \ref{alg:MVC} compute the value function of $\pi$ at a significantly faster rate than current approaches. Their running time scale as $O \left( 1/\sqrt{1-\lambda} \right)$, a significant improvement compared to $O \left( 1/(1-\lambda) \right)$ using \ref{alg:VC}.

\vspace{1mm}
\noindent \textbf{Lower bound on convergence rate of first-order algorithms.}
We present a family of hard instances of MDPs where the transient convergence of value iteration is a lower bound on the convergence rate of a large class of algorithms (including Algorithm \ref{alg:AVI}) when the number of iterations is smaller than the number of states. This is a key difference with smooth, convex optimization, where Nesterov's accelerated gradient descent is known to converge faster than gradient descent, while attaining the optimal convergence rate over the class of all smooth, convex functions.  This also shows that in all generality, accelerated methods can not provide theoretical speedups in the convergence rates, and may even diverge on some MDP instances with some specific structural properties.

{
\vspace{1mm}
\noindent \textbf{Safe Accelerated Value Iteration.} Motivated by the empirical performances of A-VI, we introduce Safe Accelerated Value Iteration (Algorithm \ref{alg:SAVI}) which combines the theoretical guarantees of VI and the numerical speedups of A-VI.  Our algorithm S-AVI interweaves aggressive A-VI steps with safe VI steps. We show that S-AVI retains the theoretical convergence guarantees of VI, thereby matching our theoretical lower bound on the convergence rate of any first-order method for MDPs.  In our numerical experiments on both structured and random instances, S-AVI performs A-VI steps more than $99 \%$ of the time and exhibits strong empirical performances,  outperforming by up to one order of magnitude the current state-of-the-art approaches.  Our framework for combining aggressive and safe algorithms into a novel algorithm that exhibits the empirical performances of the aggressive method and the convergence guarantees of the safe method is also of independent interest.}

\subsection{Related Literature.}
{Our work mainly relates to three topics of research: (i) faster algorithms for MDPs, (ii) faster computation of the Bellman operator, and (iii) MDP algorithms through the lens of optimization.}

\noindent 
\textbf{Faster algorithms for MDPs.} 
{The facts that some MDP instances require a large value of the discount factor $\lambda$ and the resulting increase in computation time is well recognized in the MDP literature, e.g.~\cite{tang2021taylor} and ~\cite{o2016combining}.
Several variants of value iteration have been considered, based on changing the update operators, either with alternate Bellman operators (such as \textit{Gauss-Seidel} and \textit{Jacobi} value iteration algorithms, see Section 6.3.3 in~\cite{Puterman})), or \textit{Anderson mixing} to speed up empirically the convergence of value iteration (see \cite{ref-c}). \cite{hendrickx-2019} consider acceleration but for the case of $\alpha$-averaged operators instead of Bellman operators.  In contrast here, we propose to change the value iteration algorithm to incorporate the idea of acceleration based on our analogy between value iteration and gradient descent and insights from accelerated algorithms for convex optimization.}

{
\noindent 
\textbf{Faster computation of the Bellman operator.}
Apart from the large value of the discount factor, another potential cause for the slow running time of \ref{alg:VI} is the estimation of the Bellman operator at every iteration.  This is particularly true when solving robust MDPs~\citep{Iyengar,Kuhn}, where the Bellman operator requires to solve large saddle-point problems for every state.  In this case, under some structural assumptions on the uncertainty sets it is possible to design fast algorithms for computing the (robust) Bellman operators, based on binary searches and homotopy methods~\citep{Iyengar,Ho}. Even for the case of nominal MDPs,  estimating the Bellman operator may become computationally expensive when the number of states and actions grow, as it requires to compute high-dimensional expectations and maximums.  Value function approximation ameliorates this issue by projecting the value function onto a lower-dimensional subspace and considering inexact Bellman updates~\citep{DeFarias,petrik2010optimization,scherrer2015approximate}.
In contrast to these methods to accelerate the computation of the Bellman operator,  we focus on designing algorithms that converge in a smaller number of iterations and any method can be used to estimate the Bellman operator.
}

{
\noindent 
\textbf{MDP algorithms through the optimization lens.}
MDPs have been studied extensively in the literature from an optimization perspective.  Recently,  {\cite{wainwright2019stochastic} and \cite{sidford2018variance} have adapted the idea of \textit{variance reduction} to Q-learning and Value Iteration, leading to more sample-efficient algorithms.}
Additionally, analogous to our connection between gradient descent and value iteration, \cite{Puterman-PI} show that policy iteration can be reformulated as a variant of the Newton's algorithm in convex optimization. \cite{protasov-PI} show that a variant of policy iteration has a global linear convergence rate, and a quadratic convergence rate when the current estimate is close enough to the optimal solution analogous to the convergence behavior of the Newton's algorithm. The connection between Newton's algorithm and policy iteration has been extended to the case of stochastic games (\cite{ref-d}). Moreover, \cite{operator-approach} draw a connection between the optimal value function of a two-player stochastic game and a policy iteration algorithm for optimizing the max-min joint spectral radius of two sets of matrices.  We refer the reader to \cite{vieillard2019connections} and \cite{grand2021convex} for reviews of the connections between MDP algorithms and the classical algorithms in convex optimization.
}

\section{Preliminaries on MDP.}\label{sec:MDP}
A Markov Decision Process (MDP) is given by a tuple $(\X,\A,\bm{P},\bm{r},\bm{p}_{0},\lambda)$, where $\X$ is the set of states and $\A$ is the state of actions.  We let $|\X|=n < + \infty,  |\A|=A < + \infty.$ We call $\bm{P} \in \R^{n \times A \times n}$ the transition kernel, $\bm{r} \in \R^{n \times A}$ the state-action reward, $\bm{p}_{0} \in \R^{n}_{+}$ the initial state distribution, and $\lambda \in (0,1)$ the discount factor. A \textit{policy} $\pi \in \R^{n \times A}_{+}$ maps each state to a probability distribution over the set of actions $\A$.  We cal $\Pi$ the set of all (stationary, Markovian) policies. For each policy $\pi$ and transition kernel $\bm{P}$, one can associate a \textit{value function} in $\R^{n}$, defined as $v^{\pi}_{i} = E^{\pi, \boldsymbol{P}} \left[ \sum_{t=0}^{\infty} \lambda^{t}r_{i_{t}a_{t}} \; \bigg| \; i_{0} = i \right], \forall \; i \in \X,$ where $(i_{t},a_{t})$ is the state-action pair visited at time $t$. The goal of the decision-maker is to compute a policy $\pi^{*}$ that maximizes the expected discounted reward, defined as $R(\pi) = \bm{p}_{0}^{\top} \bm{v}^{\pi}.$
Given a policy $\pi$, we define the reward vector $\bm{r}_{\pi} \in \R^{n}$ as $r_{\pi,i} = \sum_{a \in \A} \pi_{ia}r_{ia}, \forall \; i \in \X.$ The policy $\pi$ defines a Markov chain on $\X$, whose transition matrix is $ \bm{L}_{\pi} \in \R^{n \times n}$ defined as $ L_{\pi,ij} = \sum_{a \in \A} \pi_{ia}P_{iaj}, \forall \; (i,j) \in \X \times \X.$
As we mention earlier, several algorithms to compute an optimal policy have been studied, including value iteration, policy iteration and linear programming. We refer the reader to \cite{Puterman} for a detailed discussion.

Since our iterative algorithms use value iteration as a basic step, let us introduce it more specifically. Define the \textit{Bellman operator} $T: \R^{n} \rightarrow \R^{n}$, where for $\bm{v} \in \R^{n}$,
\begin{equation}\label{eq:T_max-def}
T(\bm{v})_{i} = \max_{a \in \A} \{ r_{ia} + \lambda \cdot \bm{P}_{ia}^{\top}\bm{v} \}, \forall \; i \in \X,
\end{equation}
{
where for a state-action pair $(i,a) \in \X \times \A$, the vector $\bm{P}_{ia}$ represents the probability distribution over the next states, i.e., $P_{iaj}$ is a probability that the agent transitions to state $j$ when action $a$ is chosen at state $i$.}
The operator $T$ is a contraction of $\left(\R^{n},\| \cdot \|_{\infty} \right)$: for any vectors $\bm{v},\bm{w} \in \R^{n}$, we have $\| T(\bm{v}) - T(\bm{w}) \|_{\infty}  \leq \lambda \cdot \| \bm{v} - \bm{w}\|_{\infty}.$ {The value iteration (VI) algorithm generates a sequence of iterates $\bm{v}_{0},\bm{v}_{1}, \bm{v}_{2},...$ as follows:}
\begin{equation}\label{alg:VI}\tag{VI}
\bm{v}_{0} \in \R^{n}, \bm{v}_{s+1} = T(\bm{v}_{s}), \forall \; s \geq 0.
\end{equation}
Following \cite{Puterman}, Chapter 6.3, the value function $\bm{v}^{*}$ of the optimal policy $\pi^{*}$ is the unique fixed-point of the operator $T$, and for any $s \geq 0,$ we have $ \| \bm{v}_{s} - \bm{v}^{*} \|_{\infty} \leq \lambda^{s} \| \bm{v}_{0} - \bm{v}^{*} \|_{\infty}$. 
Moreover,  it is known that
\begin{equation}\label{eq:eps-approx-policy}
\| \bm{v}_{s} - \bm{v}_{s+1} \|_{\infty} \leq \epsilon (1-\lambda) (2 \lambda)^{-1} \Rightarrow \| \bm{v}^{\pi_{s}} - \bm{v}^{*} \|_{\infty} \leq \epsilon,
\end{equation}
where $\bm{v}^{\pi_{s}}$ is the value function of $\pi_{s}$, the policy attaining the maximum in each row of $T(\bm{v}_{s})$ {(see for instance Theorem 6.3.3 in \cite{Puterman}).}  
 An $\epsilon$-optimal policy is a policy $\pi$ such that $\| \bm{v}^{\pi} - \bm{v}^{*} \|_{\infty} \leq \epsilon.$ {From \eqref{eq:eps-approx-policy}, an $\epsilon$-policy can be recovered after VI runs for $O\left((1-\lambda)^{-1}\log(\epsilon^{-1}(1-\lambda)^{-1})\right)$ iterations, by choosing the policy attaining the maximum on each row of $T(\bm{v}_{s})$ as in \eqref{eq:T_max-def}. Note that evaluating $T(\bm{v})$ at each iteration can be done in $n^{2}A$ arithmetic operations.} Overall, one can recover an $\epsilon$-optimal policy after a number of arithmetic operations in 
\begin{equation}\label{eq:rate-VI-eps-pol}
\mathcal{O} \left( n^{2} \cdot A \cdot \dfrac{1}{1-\lambda} \log\left(\dfrac{1}{\epsilon \cdot (1-\lambda)}\right) \right).
\end{equation}
Therefore, we know that (i) $\lim_{s \rightarrow \infty} \bm{v}_{s} = \bm{v}^{*}$, (ii) $\bm{v}_{s} = \bm{v}^{*} + \mathcal{O}\left( \lambda^{s} \right)$ and (iii) each iteration requires a number of operations in the order of $n^{2} \cdot A,$  while the number of iterations before convergence grows as $1/(1-\lambda)$. Algorithm \ref{alg:VI} does not scale well when $\lambda$ is close to $1$, which is of interest in many applications.  { For instance,  the evaluation horizon of $1000$ periods from the OpenAI gym control tasks~\citep{brockman2016openai} corresponds to a discount factor satisfying $1/(1-\lambda) = 1000$, i.e., $\lambda=0.999$. Discount factors of at least $0.99$ are also common in healthcare applications of MDPs, see some examples in Chapter 6 and Chapter 8 in \cite{boucherie2017markov}.}
{ To compare with \eqref{eq:rate-VI-eps-pol}, note that the Policy Iteration algorithm returns an optimal policy in at most $O(nA (1-\lambda)^{-1} \log((1-\lambda)^{-1})))$ iterations~\citep{scherrer2016improved}; each iteration requires to invert a matrix, which is more costly than the iteration cost of VI.}
{We also note that for \ref{alg:VI}, better convergence rates than $\lambda$ can be shown for the \textit{span seminorm} (sp) of $\bm{v}_{s} - T(\bm{v}_{s})$, where $sp(\bm{v}) = \max_{i \in \X} v_{i} - \min_{i \in \X} v_{i}$. In particular,  for  \ref{alg:VI} the span seminorm of $\bm{v}_{s} - T(\bm{v}_{s})$ converges to 0 as $O \left( \left( \gamma \lambda \right)^{s} \right)$, where $\gamma \in [0,1]$ is an upper bound on the second largest eigenvalue of the transition matrices of the MDP.  We refer the reader to Chapter 6.6 in \cite{Puterman} for more details. Interestingly, we will see in Section \ref{sec:analysis_MAVC} that the convergence properties of our accelerated algorithms is also related to the eigenvalues of the transition matrices of the MDPs.}

We also consider the problem of computing the value function $\bm{v}^{\pi}$ of a policy $\pi$. The operator associated with a policy $\pi$ {(possibly randomized)}, $T_{\pi}:\R^{n} \rightarrow \R^{n}$, is defined as
\begin{equation}\label{eq:T_pi}
T_{\pi}(\bm{v})_{i} = \sum_{a \in \A} \pi_{ia}\left(r_{ia} + \lambda \bm{P}_{ia}^{\top}\bm{v}\right), \forall \; i \in \X.
\end{equation}
{Note that the relationship between $T$ and $T_{\pi}$: the Bellman operator maximizes each component of $T_{\pi}(\bm{v})$ over the simplex, which in turns yields a deterministic choice of action at each state.}
The sequence of vectors $\left(T_{\pi}^{s}(\bm{v}_{0})\right)_{s \geq 0}$ converges to $\bm{v}^{\pi}$ with a geometric rate of $\lambda$, for any initial vector $\bm{v}_{0}$~\citep{Puterman}. Therefore, one can compute an $\epsilon$-approximation of the value function $\bm{v}^{\pi}$ in a number of iterations in \eqref{eq:rate-VI-eps-pol}, using the following Value Computation (VC) algorithm:
\begin{equation}\label{alg:VC}\tag{VC}
\bm{v}_{0} \in \R^{n}, \bm{v}_{s+1} = T_{\pi}(\bm{v}_{s}), \forall \; s \geq 0.
\end{equation}
Finally, we also need the notion of \textit{reversible} Markov chain.  { Intuitively, a Markov chain is reversible if when at equilibrium (i.e., when the chain starts at the stationary distribution),  reversing the time would yield the same Markov chain.} In particular, we have the following definition. 
\begin{definition}\label{def:reversible}
A Markov chain associated with an irreducible transition matrix $\bm{L}$ and a unique stationary distribution $\bm{\nu}$ is \textit{reversible} if and only if it satisfies the following detailed balance equation: 
$\nu_{i} L_{ij} = \nu_{j} L_{ji}, \forall \; (i,j) \in \X \times \X.$
\end{definition} 
A reversible Markov chains associated with a transition matrix $\bm{L}$ has the remarkable property that the eigenvalues of $\bm{L}$ are all real, as proved in Chapter 6, Section 2.1 in \cite{bremaud2013markov}.  This will prove useful when we analyze the rates of convergence of our accelerated algorithms in Section \ref{sec:analysis_MAVC}.
\section{Value Iteration and Gradient Descent.}\label{sec:VI-GD}
In this section we present a connection between value iteration and gradient descent. Let us first recall the main results on gradient descent.
For a given differentiable function $f: \R^{n} \rightarrow \R$ and a sequence of non-negative scalars $(\alpha_{s})_{s \geq 0}$, the gradient descent algorithm (Algorithm \ref{alg:GD}) is described as:
\begin{equation}\label{alg:GD}\tag{GD}
\bm{v}_{0} \in \R^{n}, \bm{v}_{s+1} = \bm{v}_{s}-\alpha_{s} \nabla f (\bm{v}_{s}), \forall \; s \geq 0.
\end{equation}
{
If $f$ is convex, then $\left( f(\bm{v}_{s}) \right)_{s \geq 0}$ will converge to the minimum of $f$ as soon as $\sum_{s=0}^{+ \infty} \alpha_{s} = + \infty, \lim_{s \rightarrow + \infty} \alpha_{s}=0$;} if additionally the function $f$ is $\mu$-strongly convex and $L$-Lipschitz continuous ($L>\mu >0)$, the sequence $(\bm{v}_{s})_{s \geq 0}$ produced by \ref{alg:GD} does converge at a linear rate to $\bm{v}_{f}$ the minimizer of $f$ as soon as $\alpha_{s} = \alpha \in (0,2/L], \forall \; s \geq 0$ (see Chapter 2.1.5 in \cite{nesterov-book}). Moreover, we obtain an optimal rate with the fixed step size $\alpha= 2/(L+\mu)$ for which, for $\kappa=L/\mu$, \[ \| \bm{v}_{s} - \bm{v}_{f} \|_{2}^{2} = \mathcal{O} \left( \left( \dfrac{1-1/\kappa}{1+1/\kappa} \right)^{s} \right).\]
 
In order to compute an optimal policy $\pi^{*}$, we want to compute the vector $\bm{v}^{*}$, the unique solution of $\bm{v}^{*} - T(\bm{v}^{*})=\bm{0}$.  {One of our main idea is to treat the vector $\bm{v} - T(\bm{v})=\left( \bm{I} - T \right)(\bm{v})$ as the gradient of some function from $\R^{n}$ to $\R$, applied to the vector $\bm{v}$, and we are looking for a vector $\bm{v}^{*}$ at which this gradient vanishes, i.e., for which $\bm{v}^{*} - T(\bm{v}^{*}) = \bm{0}.$ Note that this is only an analogy that we use to build novel algorithms and tune the step sizes.  In particular,  we can exactly characterize the MDP instances for which there exists a function $f: \R^{n} \rightarrow \R$ such that $\nabla f (\bm{v})=\left(\bm{I}-T\right)(\bm{v})$. In short, if $\nabla f = \bm{I}-T$, then on open sets where $T$ is differentiable we can define the Hessian of $f$. This imposes some conditions on the symmetry of the transition rates of the MDP instances.  We present the details in Appendix \ref{app:non-existence-of-f}.} Inspired by this analogy, we consider the following algorithm, defined for any choice of step sizes $(\alpha_{s})_{s \geq 0}$.
\begin{equation}\label{alg:R-VI}\tag{R-VI}
\boxed{
\bm{v}_{0} \in \R^{n}, \bm{v}_{s+1} = \bm{v}_{s}-\alpha_{s} \left( \bm{v}_{s}- T(\bm{v}_{s}) \right), \forall \; s \geq 0.}
\end{equation}
Note that this idea was also considered in \cite{relax-VI-2} and \cite{relax-VI-1} where the authors refer to it as \textit{relaxation}. Therefore, we refer to Algorithm \ref{alg:R-VI} as \textit{Relaxed Value Iteration} (R-VI). However, no convergence guarantee or formal connection to \ref{alg:GD} was provided. Additionally, Algorithm \ref{alg:R-VI} is reminiscent to Krasnoselskii-Mann (KM) iteration in non-expansive operator theory (see \cite{krasnoselskii-1955}): when $\alpha_{s} \in (0,1), \forall \; s \geq 0$ and $T$ is a non-expansive operator ($\lambda \leq 1$), Algorithm \ref{alg:R-VI} is known to converge to one of the fixed-points of $T$, but no rate is provided (\cite{krasnoselskii-1955}). Note that in the case of a non-expansive operator, this convergence result excludes the case $\alpha =1$, which recovers the original value iteration (Algorithm \ref{alg:VI}). We present the following extension to the result for non-expansive operators.
\begin{proposition}\label{prop:VI-step size} Consider $\left( \bm{v}_{s} \right)_{s \geq 0}$ the iterates of \ref{alg:R-VI} and $\bm{v}^{*}$ the unique fixed-point of the Bellman operator $T$.
\begin{enumerate}
\item Let the step sizes in Algorithm \ref{alg:R-VI} be $\alpha_{s} = \alpha \in (0,2/(1+\lambda)), \forall \; s \geq 0.$
Then
\[\| \bm{v}_{s} - \bm{v}^{*} \|_{\infty} \leq  \| \bm{v}_{0} - \bm{v}^{*} \|_{\infty} \cdot ( \lambda \cdot \alpha + | 1- \alpha |)^{s} , \forall \; s \geq 0.\]
The optimal rate is $\lambda$, attained when $\alpha=1$.
\item 
Let $(\alpha_{s})_{s \geq 0}$ be a sequence of non-negative scalars such that \\
$ \sum_{s=0}^{+\infty} \alpha_{s} = + \infty, \lim_{s \rightarrow +\infty} \alpha_{s} = 0.$
Then
\[\| \bm{v}_{s} - \bm{v}^{*} \|_{\infty} \leq \| \bm{v}_{0} - \bm{v}^{*} \|_{\infty}  \cdot \mathcal{O} \left(  \exp \left( - \sum_{\ell = 0}^{s} \alpha_{\ell} \right) \right)  , \forall \; s \geq 0.\]
\end{enumerate}
\end{proposition}
\begin{proof}
\begin{enumerate}
\item Let $\alpha \in \R_{+}$ and $s \geq 0.$ We have
\begin{align}
\| \bm{v}_{s+1} - \bm{v}^{*} \|_{\infty} & = \| (1-\alpha)\bm{v}_{s} +\alpha T(\bm{v}_{s}) - \bm{v}^{*}  \|_{\infty} \nonumber \\
& = \| (1-\alpha)(\bm{v}_{s} - \bm{v}^{*}  ) + \alpha ( T(\bm{v}_{s}) - \bm{v}^{*} ) \|_{\infty} \nonumber \\
& \leq \| (1-\alpha)(\bm{v}_{s} - \bm{v}^{*}  )\|_{\infty}  + \alpha \cdot \| T(\bm{v}_{s}) - \bm{v}^{*}  \|_{\infty} \label{eq:proof-R-VI} \\
& \leq | 1 -\alpha | \cdot \|\bm{v}_{s} - \bm{v}^{*}  \|_{\infty} +  \alpha \cdot  \| T(\bm{v}_{s}) - T(\bm{v}^{*})   \|_{\infty} \label{eq:proof-VI-2} \\
& \leq | 1 -\alpha | \cdot \|\bm{v}_{s} -\bm{v}^{*}   \|_{\infty} +  \alpha \cdot \lambda \cdot  \|  \bm{v}_{s} - \bm{v}^{*}  \|_{\infty} \label{eq:proof-VI-3}\\
& \leq ( \lambda \cdot \alpha + | 1- \alpha |) \cdot \| \bm{v}_{s} - \bm{v}^{*}  \|_{\infty},
\end{align}
where \eqref{eq:proof-R-VI} follows from triangle inequality, \eqref{eq:proof-VI-2} follows from $\bm{v}^{*} = T(\bm{v}^{*})$, and \eqref{eq:proof-VI-3} follows from the Bellman operator being a contraction with factor $\lambda$. We can therefore conclude that
\[
\| \bm{v}_{s} - \bm{v}^{*} \|_{\infty} \leq ( \lambda \cdot \alpha + | 1- \alpha |)^{s} \cdot \| \bm{v}_{0} - \bm{v}^{*}  \|_{\infty}, \forall \; s \geq 0.
\]
Now we can see that the function $ \alpha \mapsto ( \lambda \cdot \alpha + | 1- \alpha |)$ remains strictly smaller than $1$ only for $\alpha \in (0,2/(1+\lambda)),$ and attains its minimum, $\lambda$, for $\alpha =1$.
\item Since $\lim_{s \rightarrow \infty} \alpha_{s} = 0$, there exists an integer $s_{0} \in \N$ such that $\alpha_{s_{0}} \leq 1, \forall \; s \geq s_{0}.$ The same reasoning as above gives
\[ \| \bm{v}_{s} - \bm{v}^{*} \|_{\infty} \leq \left( \prod_{\ell=s_{0}}^{s-1} ( 1 - (1-\lambda)\alpha_{\ell})) \right) \cdot \| \bm{v}_{s_{0}} - \bm{v}^{*}  \|_{\infty}, \forall \; s \geq s_{0}. \]
Finally,  { from the convexity of the exponential function, we have $1-x \leq e^{-x}$ for any $x \in \R$ and therefore} \[\prod_{\ell=s_{0}}^{s-1} ( 1 - (1-\lambda)\alpha_{\ell}))  = \mathcal{O} \left(  \exp \left( - \sum_{\ell = s_{0}}^{s} \alpha_{\ell} \right) \right)=\mathcal{O} \left(  \exp \left( - \sum_{\ell = 0}^{s} \alpha_{\ell} \right) \right).\]
\end{enumerate}
\end{proof}
Several remarks are in order.
{
\begin{remark}[Optimal rate for \ref{alg:R-VI}]
Note that the optimal theoretical convergence rate of \ref{alg:R-VI} is $\lambda$, attained for $\alpha_{s} = 1, \forall \; s \geq 0$.  Therefore, \ref{alg:R-VI} does not provide a better theoretical convergence rate than \ref{alg:VI}. However, Proposition \ref{prop:VI-step size} shows that the convergence properties of \ref{alg:GD} can be extended to \ref{alg:R-VI}, with analogous convergence rates and choices of step sizes. In the next section we will build upon this analogy to derive novel accelerated algorithms for MDPs and choose the step sizes.
\end{remark}}
\begin{remark}[Some other contractions]
Note that Proposition \ref{prop:VI-step size} holds for any operator that is a contraction for $\| \cdot \|_{\infty}$. Therefore, it can also be applied to compute the value function of a policy $\bm{v}^{\pi}$, which is the fixed-point of $T_{\pi}$ \eqref{eq:T_pi}.  {Other contracting operators of interest include softmax and log-sum-exp operators ~\citep{boltzman-1}},  and robust Bellman operators for Robust MDPs (see \cite{Iyengar},\cite{Kuhn}, \cite{GGC}).
\end{remark}
We discuss here the choice of the upper bound $2/(1+\lambda)$ for the step size. Recall that $T$ is a contraction operator with factor $\lambda$. From the triangle inequality we have, for any vectors $\bm{v},\bm{w} \in \R^{n},$
\begin{align}
(1-\lambda) \cdot \| \bm{v} - \bm{w} \|_{\infty} & \leq \| (\bm{I} - T)(\bm{v}) - (\bm{I} - T)(\bm{w}) \|_{\infty},
\label{eq:mu-infty} \\
\| (\bm{I} - T)(\bm{v}) - (\bm{I} - T)(\bm{w}) \|_{\infty} & \leq (1+\lambda) \cdot \| \bm{v} - \bm{w} \|_{\infty}. \label{eq:L-infty}
\end{align}
Note that \eqref{eq:mu-infty}-\eqref{eq:L-infty} are the analogous of the following inequalities for differentiable, $\mu$-strongly convex, $L$-Lipschitz continuous function $f: \R^{n} \rightarrow \R$: for all vectors $\bm{v},\bm{w} \in \R^{n},$
\begin{align}
\mu \cdot \| \bm{v} - \bm{w} \|_{2} & \leq \| \nabla f(\bm{v}) - \nabla f (\bm{w}) \|_{2}, \label{eq:mu-2} \\ 
\| \nabla f(\bm{v}) - \nabla f (\bm{w}) \|_{2} & \leq L \cdot \| \bm{v} - \bm{w} \|_{2}. \label{eq:L-2}
\end{align}
{
Based on this analogy, we write here $\mu = 1- \lambda, L=1+\lambda.$} Note that in Proposition \ref{prop:VI-step size}, the maximum fixed step size guaranteeing convergence is $2/(1+\lambda)=2/L$, and that the optimal rate is $\lambda=(1+\lambda - (1 -\lambda))/(1+\lambda + 1 - \lambda) = (L-\mu)/(L + \mu)$, attained for $\alpha=1=2 / (1 + \lambda + 1 - \lambda) = 2 / (L + \mu).$ Therefore, the properties of \ref{alg:GD} for a function satisfying \eqref{eq:mu-2}-\eqref{eq:L-2} do translate for Algorithm \ref{alg:R-VI} with Bellman operator $T$ satisfying \eqref{eq:mu-infty}-\eqref{eq:L-infty}.
Similarly, Proposition \ref{prop:VI-step size} is reminiscent to the following conditions of convergence for gradient descent with varying step size:
$ \sum_{n=0}^{+\infty} \alpha_{n} = + \infty, \sum_{n=0}^{+\infty} \alpha_{n}^{2} < + \infty. $ 
{Our analogy between \eqref{eq:mu-infty}-\eqref{eq:L-infty} and \eqref{eq:mu-2}-\eqref{eq:L-2} will prove useful to derive the step sizes of our accelerated algorithms in the next section.}

\begin{remark}[Comparison with the proof for \ref{alg:GD}]
We would like to emphasize that the convergence proof for \ref{alg:GD} does not readily extend to Algorithm \ref{alg:R-VI}. In particular, there is a fundamental difference between $\| \cdot\|_{\infty}$ and $\| \cdot \|_{2}$, since $\| \cdot \|_{2}$ is naturally related to the scalar product of $\R^{n}$, i.e., $\| \bm{v} \|_{2}^{2} = \bm{v}^{\top}\bm{v}$, whereas this is not the case for $\| \cdot \|_{\infty}.$ Therefore, we can not rely on scalar products and Taylor expansions of first order to prove our convergence results. Moreover, if one were to infer some new constants $\mu,L$ for \eqref{eq:mu-2}-\eqref{eq:L-2} from \eqref{eq:mu-infty}-\eqref{eq:L-infty} and the equivalence of norms in finite dimension, one would obtain (in general) $\mu = (1-\lambda)/\sqrt{n},L=\sqrt{n}(1+\lambda)$, essentially loosing a factor of $n$ as regards to the convergence rate guarantee of Algorithm \ref{alg:R-VI}. 
\end{remark}
\section{Accelerated and Momentum Value Iteration.}\label{sec:introducing-acc-mom-vi}
{In this section, we build upon the connection between \ref{alg:VI} and \ref{alg:GD} to introduce accelerated and momentum algorithms for computing the fixed-point of the Bellman operator $T$.}
\subsection{Accelerated Value Iteration.}
Accelerated Gradient Descent (A-GD, \cite{nesterov-1983} and \cite{nesterov-book}) has been extended to popular iterative methods such as FISTA (\cite{fista}) and F-ADMM (\cite{fadmm}). Building upon our analogy between gradient descent and value iteration, we extend A-GD to an accelerated iterative algorithm for computing $\bm{v}^{*}$, the unique fixed-point of the contraction $T$. In particular,  {for any sequences of step sizes} $(\alpha_{s})_{s \geq 0}$ and $(\gamma_{s})_{s \geq 0} \in \R^{\N}$,  we propose the following Accelerated Value Iteration \eqref{alg:AVI} algorithm.
\begin{equation}\label{alg:AVI}\tag{A-VI}
\boxed{
\bm{v}_{0},\bm{v}_{1} \in \R^{n}, \begin{cases}
    \bm{h}_{s}=\bm{v}_{s}+\gamma_{s}\cdot \left( \bm{v}_{s}-\bm{v}_{s-1} \right), \\
	\bm{v}_{s+1} \gets \bm{h}_{s}-\alpha_{s} \left( \bm{h}_{s}- T \left( \bm{h}_{s} \right) \right), \end{cases}\forall \; s \geq 1.}
\end{equation}

\vspace{1mm}
\noindent {\textbf{Choice of constant step sizes.}  We base our choice of step sizes on our connection between value iteration and gradient descent. In particular, \cite{nesterov-book} (Algorithm 2.2.11) chooses $\alpha=1/L,\gamma = (\sqrt{L}-\sqrt{\mu})/(\sqrt{L}+\sqrt{\mu})$. From our analogy between \ref{alg:VI} and \ref{alg:GD}, and from \eqref{eq:mu-infty}-\eqref{eq:L-infty} and \eqref{eq:mu-2}-\eqref{eq:L-2}, we have $\mu= (1- \lambda), L=(1+ \lambda),$ which corresponds to our first choice of fixed step sizes of
\begin{equation}\label{eq:tuning-1}
\alpha_{s}  = \alpha = 1/(1+\lambda),
 \gamma_{s} = \gamma = \left(1-\sqrt{1- \lambda^{2}}\right)/\lambda, \forall s \; \geq 1.
 \end{equation}
 \begin{remark}
 In smooth convex optimization, it can be challenging to estimate the value of the Lipschitz-constant $L$ (respectively, of the strong-convexity constant $\mu$), thereby making it difficult to find the optimal step sizes. \cite{estimation-mu} proposes backtracking schemes to evaluate optimal step sizes. In contrast, for our algorithm, the Lipschitz-constant is replaced by $(1+\lambda)$ (respectively, the strong-convexity parameter is replaced by $(1-\lambda)$), and it is the discount factor $\lambda$ that needs to be evaluated. The value of the discount factor can be seen as a choice of the decision-maker, depending on the impact of future rewards, i.e., depending on the effective time horizon.
 \end{remark}
\subsection{Momentum Value Iteration.}
Momentum Gradient Descent is a popular method to improve the convergence of Algorithm \ref{alg:GD} (\citet{polyak-1964}, \citet{heavyball-2015}). For any sequences of scalars $(\alpha_{s})_{s \geq 0}$ and $(\beta_{s})_{s \geq 0} \in \R^{\N}$,  we propose the following Momentum Value Iteration \eqref{alg:MVI} algorithm.
\begin{equation}\label{alg:MVI}\tag{M-VI}
\boxed{
\bm{v}_{0},\bm{v}_{1} \in \R^{n},
	\bm{v}_{s+1} =\bm{v}_{s}-\alpha_{s} \left( \bm{v}_{s}- T \left( \bm{v}_{s} \right) \right)  + \beta_{s}\cdot \left( \bm{v}_{s}-\bm{v}_{s-1} \right) , \forall \; s \geq 1.}
\end{equation}
\noindent {\textbf{Choice of constant step sizes.} The authors in \citet{polyak-1964} and \citet{heavyball-2015} show that if a smooth, strongly convex function is twice differentiable, the optimal choice of constant for Momentum GD is
 $\alpha=4/(\sqrt{L} + \sqrt{\mu})^{2},\beta = (\sqrt{L}-\sqrt{\mu})^{2}/(\sqrt{L}+\sqrt{\mu})^{2}$. We use $\mu= (1- \lambda), L=(1+ \lambda)$ based on our analogy between \ref{alg:VI} and \ref{alg:GD}, and from \eqref{eq:mu-infty}-\eqref{eq:L-infty} and \eqref{eq:mu-2}-\eqref{eq:L-2}. Therefore, we consider the following choice of step sizes.
 \begin{equation}\label{eq:tuning-MVC-1}
\alpha_{s}=\alpha  =2/(1+\sqrt{1-\lambda^{2}}), 
\gamma_{s}=\gamma  =(1-\sqrt{1-\lambda^{2}})/(1+\sqrt{1-\lambda^{2}}), \forall \; s \geq 1.
\end{equation} 
\section{Accelerated and Momentum for affine operators: the case of Value Computation.}\label{sec:analysis_MAVC}
{As a first step toward analyzing \ref{alg:AVI} and \ref{alg:MVI}, we focus on the specific problem of computing the value function $\bm{v}^{\pi}$ of a policy $\pi$.  This is equivalent to computing the fixed-point of the \textit{affine} operator $T_{\pi}$ introduced in \eqref{eq:T_pi}.   Recall that a classical fixed-point iteration scheme such as Algorithm \ref{alg:VC} converges to $\bm{v}^{\pi}$ as $O(\lambda^{s})$ after $s$ iterations.  In this section, we consider Accelerated Value Computation (Algorithm \eqref{alg:AVC}) and Momentum Value Computation (Algorithm \eqref{alg:MVC}) and analyze their convergence properties.
\begin{equation}\label{alg:AVC}\tag{A-VC}
\bm{v}_{0},\bm{v}_{1} \in \R^{n}, \begin{cases}
    \bm{h}_{s}=\bm{v}_{s}+\gamma_{s}\cdot \left( \bm{v}_{s}-\bm{v}_{s-1} \right), \\
	\bm{v}_{s+1} \gets \bm{h}_{s}-\alpha_{s} \left( \bm{h}_{s}- T_{\pi} \left( \bm{h}_{s} \right) \right), \end{cases}\forall \; s \geq 1.
\end{equation}
\begin{equation}\label{alg:MVC}\tag{M-VC}
\bm{v}_{0},\bm{v}_{1} \in \R^{n},
	\bm{v}_{s+1} =\bm{v}_{s}-\alpha_{s} \left( \bm{v}_{s}-T_{\pi} \left( \bm{v}_{s} \right) \right) + \beta_{s}\cdot \left( \bm{v}_{s}-\bm{v}_{s-1} \right)  , \forall \; s \geq 1.
\end{equation}}
\subsection{Analysis of Accelerated Value Computation.}\label{sec:analysis_AVC}
In this section we provide a theoretical convergence rate for Algorithm \ref{alg:AVC} for some specific MDP instances.
\subsubsection{Convergence rate}
Consider an MDP instance and a policy $\pi$.  We will show that Algorithm \ref{alg:AVC} returns the value function $\bm{v}^{\pi}$ significantly faster than \ref{alg:VC}, if the Markov chain defined by the policy $\pi$ is irreducible and reversible (see Definition \ref{def:reversible}). 
The irreducibility condition is very general and to ensure that the Markov chain has a unique stationary distribution. 
{The reversibility assumption has appeared in some applications of Reinforcement Learning~\citep{garnet-1,revers-2}, although it is more widespread in Markov Chain Monte Carlo methods \citep{revers-3} and random walks on graph \citep{rd-walk-1,bremaud2013markov}. }

{We present our main result on the convergence rate of \ref{alg:AVC} in the next theorem.}

\begin{theorem}\label{th:AVI-value-computation}
Consider an MDP and a policy $\pi$ defining an irreducible, reversible Markov chain on the set of states $\X$.
Consider $(\bm{v}_{s})_{s \geq 0}$ the sequence of iterates of Algorithm \ref{alg:AVC} with step sizes \eqref{eq:tuning-1}.
Then $(\bm{v}_{s})_{s \geq 0}$ converges to $\bm{v}^{\pi}$.
Moreover, for any $  \eta>0 $, for $\kappa = (1+\lambda)/(1-\lambda)$,
\[\bm{v}_{s} = \bm{v}^{\pi} + o\left( \left(1- \sqrt{\dfrac{1}{\kappa}} + \eta \right)^{s}\right), \forall \; s \geq 0.\]
Additionally,  we always have $0 < 1-\sqrt{1/\kappa} < \lambda$ for $\lambda \in (0,1)$, i.e., \ref{alg:AVC} enjoys better convergence guarantees than \ref{alg:VC}.
\end{theorem} 
 We first analyze the structure of the iterations in Algorithm \ref{alg:AVC}. Let $\bm{I} \in \R^{n \times n}$ be the identity matrix and $ \bm{B}_{\pi}  \in \R^{2n \times 2n}$ defined as
\begin{equation}\label{eq:def-B-pi}
 \bm{B}_{\pi} = \left[ {\begin{array}{cc}
(1-\sqrt{1/\kappa} ) \left( \bm{I} +\bm{L}_{\pi}\right)  &- (1-\sqrt{1/\kappa})^{2}/2\left( \bm{I} +  \bm{L}_{\pi} \right) \\
    \bm{I}  & \bm{0}
  \end{array} } \right].
\end{equation}
\begin{lemma}\label{lem:LTV-AVC} 
Let $\bm{x}^{\pi} = (\bm{v}^{\pi},\bm{v}^{\pi}) \in \R^{2 n},\bm{x}_{s} = (\bm{v}_{s},\bm{v}_{s-1}) \in \R^{2 n}$. Then $$\bm{x}_{s} = \bm{x}^{\pi}+ \bm{B}_{\pi}^{s-1}\left(\bm{x}_{1}-\bm{x}^{\pi}\right), \forall \; s \geq 1.$$
\end{lemma}
\begin{proof}
Let $(\bm{v}_{s})_{s \geq 0}$ be the sequence of iterates of Algorithm \ref{alg:AVC}. We have
\begin{align*}
\bm{v}_{s+1} & = (1-\alpha) \bm{h}_{s} + \alpha T_{\pi}(\bm{h}_{s}) \\
&= (1-\alpha)(1+\gamma)\bm{v}_{s} - (1- \alpha) \gamma \bm{v}_{s-1} + \alpha (1+\gamma) \lambda \bm{L}_{\pi}\bm{v}_{s} - \alpha \gamma \lambda \bm{L}_{\pi}\bm{v}_{s-1} + \alpha \bm{r}_{\pi} \\
& = \left( (1-\alpha)(1+\gamma) \bm{I} + \alpha (1+\gamma)\lambda \bm{L}_{\pi} \right) \bm{v}_{s}
  - \left( (1-\alpha)\gamma \bm{I} + \alpha \gamma \lambda \bm{L}_{\pi} \right) \bm{v}_{s-1} + \alpha \bm{r}_{\pi}.
\end{align*}
Recall that $\alpha = 1/(1+\lambda),\gamma = (1-\sqrt{1-\lambda^{2}})/\lambda$ and that $\kappa = (1+\lambda)/(1-\lambda).$ Then
\begin{align*}
1 + 1 / \kappa - 2 \alpha&  = 0, \\
(1-\alpha)(1+\gamma)  = \alpha \lambda  (1+\gamma ) & = 1 - \sqrt{1 / \kappa}, \\
\alpha \gamma \lambda  = (1-\alpha)\gamma  = \alpha - \sqrt{1 / \kappa} & = \dfrac{1}{2} \left(1 - \sqrt{1 / \kappa}\right)^{2}.
\end{align*}
Therefore, for $\bm{b}_{\pi} = \left(\bm{r}_{\pi},\bm{r}_{\pi} \right)$,
\begin{equation}\label{eq:rec-x-AVC}
\bm{x}_{s+1} = \bm{B}_{\pi}\bm{x}_{s} + \bm{b}_{\pi}.
\end{equation}
Note that $\bm{x}^{\pi}=\bm{B}_{\pi}\bm{x}^{\pi}+\bm{b}_{\pi}$. Therefore, for $s \geq 0$,
$$ \bm{x}_{s+1} - \bm{x}^{\pi} = \bm{B}_{\pi}\bm{x}_{s}+\bm{b}_{\pi} - \bm{B}_{\pi}\bm{x}^{\pi}-\bm{b}_{\pi} = \bm{B}_{\pi}\left( \bm{x}_{s}-\bm{x}^{\pi}\right).$$
From this we conclude 
$\bm{x}_{s} = \bm{x}^{\pi}+ \bm{B}_{\pi}^{s-1}\left(\bm{x}_{0}-\bm{x}^{\pi}\right), \forall \; s \geq 1.$
\end{proof}
In the next lemma we bound the spectral radius of $\bm{B}_{\pi}$.
\begin{lemma}\label{lem:radius} Consider a policy $\pi$ which defines an irreducible, reversible Markov chain on the set of states $\X$. Then we have
\[\rho(\bm{B}_{\pi}) = 1 - \sqrt{1/\kappa}.\]
\end{lemma}
\begin{proof}
{Let $\bm{y}=(\bm{y}_{1},\bm{y}_{2})$ an eigenvector of $\bm{B}_{\pi}$ and $\omega \in \C$ the corresponding eigenvalue. The equation $\bm{B}_{\pi}\bm{y} = \omega \bm{y}$ implies
\begin{align*}
(1-\sqrt{1/\kappa} ) \left( \bm{I} +\bm{L}_{\pi}\right) \bm{y}_{1} - \frac{(1-\sqrt{1/\kappa})^{2}}{2}\left( \bm{I} +  \bm{L}_{\pi} \right)\bm{y}_{2} & = 
 \omega \bm{y}_{1}, \\
\bm{y}_{1}  & = \omega \bm{y}_{2}.
\end{align*}
This leads to 
\[ \omega^{2}\bm{y}_{2} = (1-\sqrt{1/\kappa} ) \left( \bm{I} +\bm{L}_{\pi}\right) \omega \bm{y}_{2} -  \frac{(1-\sqrt{1/\kappa})^{2}}{2}\left( \bm{I} +  \bm{L}_{\pi} \right)\bm{y}_{2}.\]
Therefore,  there exists an eigenvalue $\mu$ in the spectrum of $\bm{L}_{\pi}$ such that $\omega$ satisfies the following equation:
}
\begin{equation}\label{eq:sec-order-nu-2}
  \omega^{2} - (\mu+1)(1-\sqrt{1 / \kappa})\omega + \dfrac{1}{2} (\mu+1)(1 - \sqrt{1 / \kappa})^{2} = 0.
  \end{equation}
  The discriminant of \eqref{eq:sec-order-nu-2} is $\Delta = (\mu^{2}-1) \cdot (1-\sqrt{1 / \kappa})^{2},$ which leads to the expression of the roots of \eqref{eq:sec-order-nu-2}:
  \begin{align}
  \omega^{+}(\mu) & = (1-\sqrt{1 / \kappa}) \cdot \dfrac{1}{2} \cdot \left( \mu+1 + \sqrt{\mu^{2}-1} \right),\label{eq:w_plus}\\
    \omega^{-}(\mu) & = (1-\sqrt{1 / \kappa}) \cdot \dfrac{1}{2} \cdot \left( \mu+1 - \sqrt{\mu^{2}-1} \right). \label{eq:w_minus}
  \end{align}
  { To prove $\rho(\bm{B}_{\pi}) = 1 - \sqrt{1/\kappa}$,  
  we now prove that for any eigenvalue $\mu$ of $\bm{L}_{\pi}$, \[|\omega^{+}(\mu)| \leq 1 - \sqrt{1/\kappa},  |\omega^{-}(\mu)| \leq 1 - \sqrt{1/\kappa}.\]  Let $Sp(\bm{L}_{\pi})$ the spectrum of the matrix $\bm{L}_{\pi}$. Based on \eqref{eq:w_plus}-\eqref{eq:w_minus}, it is enough to prove that for all $\mu \in Sp(\bm{L}_{\pi})$,
  \begin{align*}
  | \dfrac{1}{2} \cdot \left( \mu+1 + \sqrt{\mu^{2}-1} \right) | & \leq 1, \\
    |\dfrac{1}{2} \cdot \left( \mu+1 - \sqrt{\mu^{2}-1} \right)| & \leq 1.
  \end{align*}
  \[ \]
   First, since $\bm{L}_{\pi}$ is a stochastic matrix,  we have $ | \mu | \leq 1$.   Now $Sp(\bm{L}_{\pi}) \subset \R$ since the Markov chain defined by $\pi$ is reversible (see the beginning of Chapter 6, Section 2.1 in \cite{bremaud2013markov}).  This means that $Sp(\bm{L}_{\pi}) \subset [-1,1]$.}
 Now for $\mu \in [-1,1]$,
%\begin{align*}
%\dfrac{1}{2} | \mu+1 + \sqrt{\mu^{2}-1} | & = \dfrac{1}{2} \sqrt{(\mu+1)^{2} + \left( \sqrt{1-\mu^{2}} \right)^{2} } \\
%& = \dfrac{1}{2}\sqrt{\mu^{2}+1 + 2 \mu + 1 - \mu^{2}} \\
%& = \dfrac{1}{2}\sqrt{2 + 2 \mu } \\
%& = \dfrac{1}{2} \sqrt{2} \cdot \sqrt{1 + \mu}\\
%& \leq 1,
%\end{align*}
\[
\dfrac{1}{2} | \mu+1 + \sqrt{\mu^{2}-1} |  = \dfrac{1}{2} \sqrt{(\mu+1)^{2} + \left( \sqrt{1-\mu^{2}} \right)^{2} } 
= \dfrac{1}{2} \sqrt{2} \cdot \sqrt{1 + \mu}  \leq 1.
\]
Similarly, for $ \mu \in [-1,1],$ we have
$
\dfrac{1}{2} | \mu+1 - \sqrt{\mu^{2}-1} | \leq 1.
$
Therefore, the leading eigenvalue of $\bm{B}_{\pi}$ is attained for $\mu=1$, which leads to $\Delta=0$, and a maximum of eigenvalue of $\bm{B}_{\pi}$ equal to $ 1- \sqrt{1 / \kappa}.$
\end{proof}
\begin{remark}
Note that the reversibility of the Markov chain associated with $\pi$ is only a \textit{sufficient} condition for the spectral radius of $\bm{B}_{\pi}$ to be equal to $1- \sqrt{1/\kappa}$, but not necessary. In particular, another sufficient condition is $Sp(\bm{L}_{\pi}) \subset \R$.
\end{remark}
We are now ready to prove Theorem \ref{th:AVI-value-computation}.
\begin{proof}[Proof of Theorem \ref{th:AVI-value-computation}.]
From Lemma \ref{lem:LTV-AVC}, we can write $\bm{x}_{s} = \bm{x}^{\pi}+ \bm{B}_{\pi}^{s-1}\left(\bm{x}_{0}-\bm{x}^{\pi}\right), \forall \; s \geq 1.$ We have assumed that $\pi$ defines an irreducible, reversible Markov chain on the set of states $\X$. From Lemma \ref{lem:radius} we conclude that 
$\rho(\bm{B}_{\pi}) = 1 -\sqrt{1/\kappa}<1$. This implies that for any $\eta >0$, we have \[\lim_{s \rightarrow + \infty}\dfrac{1}{\left(1 -\sqrt{1/\kappa} + \eta \right)^{s}} \cdot \bm{B}_{\pi}^{s}\left(\bm{x}_{0}-\bm{x}^{\pi}\right) = 0.\]
We can therefore conclude that for any $\eta >0, \bm{v}_{s} = \bm{v}^{\pi} + o\left(\left(1 -\sqrt{1/\kappa} + \eta \right)^{s}\right), \forall \; s \geq 0.$

{Finally, note that $1-\sqrt{1/\kappa}$ is always smaller than $\lambda$ since
\begin{equation}\label{eq:comparison-rate}
 \lambda = 1 - \sqrt{1 / \kappa}+ \dfrac{\sqrt{1-\lambda} \cdot \left( 1-\sqrt{1-\lambda^{2}} \right) }{\sqrt{1+\lambda}}
\end{equation}
and $\lambda \in (0,1) \Rightarrow \sqrt{ 1-\lambda } \cdot \left( 1-\sqrt{1-\lambda^{2}} \right)  > 0$. }
\end{proof}
Therefore, we prove theoretical guarantees for the convergence of Algorithm \ref{alg:AVC}, under the assumption that the associated Markov chain is irreducible and reversible. In this case, Algorithm \ref{alg:AVC} converges geometrically at a rate of $1-\sqrt{1/\kappa}$, which coincidentally also corresponds to the known convergence rate for Nesterov's acceleration in the case of a strongly convex function (\cite{nesterov-book}, Algorithm 2.2.11).  
From Theorem \ref{th:AVI-value-computation} and the stopping criterion \eqref{eq:eps-approx-policy}, we can conclude that Algorithm \ref{alg:AVC} returns an $\epsilon$-approximation of $\bm{v}^{\pi}$ in a number of iterations in \[\mathcal{O} \left( n^{2} \cdot A \cdot \dfrac{1}{\sqrt{1-\lambda}}\cdot \log\left(\dfrac{1}{\epsilon \cdot (1-\lambda)}\right) \right),\]
which is significantly faster than \eqref{eq:rate-VI-eps-pol}, the convergence rate of Algorithm \ref{alg:VC}, for $\lambda$ approaching $1$.

{Even though our analysis in this section is based on the \textit{affine} property of the operator $T_{\pi}$ and restricted to a specific class of MDP instances, it provides important insights on the potential speedups achieved by accelerated algorithms.  Our analysis of \ref{alg:AVI}, where the Bellman operator $T$ is not affine, will be presented in Section \ref{sec:analysis_MAVI}.}

\subsubsection{Discussion on the irreducibility and reversibility assumption.}\label{sec:disc-reversibility}
We provide here an example of an MDP instance and a policy which do not define a reversible Markov chain and where Algorithm \ref{alg:AVC} may diverge. Consider an MDP with $n$ states where $n$ is even and one action for all states. The (only) transition matrix is defining a cycle, i.e., $L_{i,i+1}=1$ for $s \in \{1,...,n-1\}$ and $L_{n,1}=1.$ The eigenvalues of the transition matrix are the $n-th$ roots of $1$ in $\C$. 
   Now from \eqref{eq:w_plus}-\eqref{eq:w_minus}, we know that the spectral radius of the associated $\bm{B}_{\pi}$ is strictly higher than $1$ and therefore Algorithm \ref{alg:AVC} may diverge, since for $\mu=i \in Sp(\bm{L})$ we have
\[
\dfrac{1}{2}| \mu+1 + \sqrt{\mu^{2}-1} |  = | i + 1 + \sqrt{i^{2}-1}|  =\dfrac{1}{2}| i + 1 + \sqrt{-2} | = \dfrac{1}{2}\sqrt{4 + 2 \sqrt{2}} >1.
 \]
 \subsection{Analysis of Momentum Value Computation.}\label{sec:analysis-MVC}
In this section we provide a theoretical convergence rate for Algorithm \ref{alg:MVC} with the constants \eqref{eq:tuning-MVC-1}.  As the analysis is similar to \ref{alg:AVC},  we present a detailed proof in Appendix \ref{app:proof-section-MVC}.
\begin{theorem}\label{th:MVC-1}
Consider an MDP and a policy $\pi$ which defines an irreducible, reversible Markov chain on the set of states $\X$.
Consider $(\bm{v}_{s})_{s \geq 0}$ the sequence of iterates of Algorithm \ref{alg:MVC} with step sizes \eqref{eq:tuning-MVC-1}.
Then $(\bm{v}_{s})_{s \geq 0}$ converges to $\bm{v}^{\pi}$.
Moreover, for any $  \eta>0 $,
\[\bm{v}_{s} =\bm{v}^{\pi} + o \left( \left( \dfrac{1-\sqrt{1/\kappa}}{1+\sqrt{1/\kappa}} + \eta \right)^{s} \right), \forall \; s \geq 0.\]
Additionally,  we always have $0 < (1-\sqrt{1/\kappa})/(1+\sqrt{1/\kappa}) < \lambda$ for $\lambda \in (0,1)$, i.e., \ref{alg:MVC} enjoys better convergence guarantees than \ref{alg:VC}.
\end{theorem}
%\begin{remark} The constants \eqref{eq:tuning-MVC-2} are known to be the optimal ones for twice differentiable, smooth, strongly convex functions (\citet{heavyball-2015}). Therefore, it is not surprising that Algorithm M-VC with constants as in \eqref{eq:tuning-MVC-2} outperforms Algorithm M-VC with constants as in \eqref{eq:tuning-MVC-1}.
%\end{remark}
In the case where the associated Markov chain is irreducible and reversible, Algorithm \ref{alg:MVC} converges geometrically at a rate of $(1-\sqrt{1/\kappa})/(1+\sqrt{1/\kappa})$. This rate coincides with the known convergence rate for Momentum Gradient Descent in the case of a twice differentiable, strongly convex function \citep{heavyball-2015}. We would like to note that we consider $(\bm{I}-T_{\pi})$ as the gradient of a function and that $\bm{v} \mapsto (\bm{I} - T_{\pi})(\bm{v})$ is itself a differentiable function. Moreover, $(1-\sqrt{1/\kappa})/(1+\sqrt{1/\kappa})$ is always strictly smaller than $\lambda$, the known convergence rate for Algorithm \ref{alg:VC}.
Using the stopping criterion \eqref{eq:eps-approx-policy}, we can conclude that Algorithm \ref{alg:MVC} with constants \eqref{eq:tuning-MVC-1} returns an $\epsilon$-approximation of $\bm{v}^{\pi}$ in a number of operations in 
$\mathcal{O} \left( n^{2} \cdot A \cdot \dfrac{1}{\sqrt{1-\lambda}}\cdot \log\left(\dfrac{1}{\epsilon \cdot (1-\lambda)}\right) \right).$
As for \ref{alg:AVC}, this is significantly faster than \eqref{eq:rate-VI-eps-pol}, the convergence rate of Algorithm \ref{alg:VC}, for $\lambda$ approaching $1$.

\vspace{1mm}
\noindent {\textbf{Comparison with Accelerated Value Computation.}
 Accelerated Value Computation converges linearly to $\bm{v}^{\pi}$ with a rate of $1-\sqrt{1/\kappa}$. Therefore, Momentum Value Computation with constants \eqref{eq:tuning-MVC-1} has slightly better convergence guarantees than Accelerated VC, namely a rate of $(1-\sqrt{1/\kappa})/(1+\sqrt{1/\kappa})$, a situation similar to the case of convex optimization~\citep{heavyball-2015}, { where Momentum GD achieves better convergence rate than Accelerated GD for instances that are at least twice differentiable and convex}.
 
\vspace{1mm}
\noindent {\textbf{Irreducibility and Reversibility Assumption.}
We show in Section \ref{sec:analysis_AVC} that on a non-reversible MDP instance, Algorithm \ref{alg:AVC} may diverge. For the same choice of non-reversible MDP instance, Algorithm \ref{alg:MVC} may also diverge. This highlights the key role of reversibility assumption in the convergence of Theorem \ref{th:MVC-1}. We present the detail about the divergence of \ref{alg:MVC} on this MDP instance in Appendix \ref{app:proof-section-MVC}.
\section{Analysis of Accelerated and Momentum Value Iteration.}\label{sec:analysis_MAVI}
{
In the previous section we have shown that acceleration and momentum may provide stronger convergence guarantees than classical algorithms,  for the special case of affine operators and reversible MDP instances. 
In this section, we analyze the structure of the iterations of \ref{alg:AVI} and \ref{alg:MVI}. While for \ref{alg:AVI} we are not able to give a theoretical convergence rate analogous to \ref{alg:AVC}, we highlight the challenges in the analysis, including a lower bound on the performance of any \textit{first-order} algorithm for solving MDPs. }
\subsection{Analysis of Accelerated Value Iteration.}\label{sec:analysis_AVI}
In this section we present our analysis of Algorithm \ref{alg:AVI}. 
\subsubsection{Structure of the iterations.} The Bellman operator $T$ is a piecewise affine operator. Therefore, at any step $s \geq 1$, the iterate $\bm{v}_{s+1}$ is some affine function (which may change from iteration to iteration) of $\R^{2 n}$ applied  to $\bm{x}_{s}=(\bm{v}_{s},\bm{v}_{s-1})$. This leads to the following proposition, which gives a Linear Time-Varying (LTV) dynamical system formulation for the evolution of the sequence $(\bm{v}_{s})_{s \geq 0}$ (e.g.,  Section 1.3 in \cite{jungers-book}). 
In particular, we have the following proposition. We give a detailed proof in Appendix \ref{app:prop-B-pi-IRU}.
\begin{proposition}\label{prop:B-pi-IRU} {Consider $(\bm{v}_{s})_{s \geq 0}$ the sequence of iterates of Algorithm \ref{alg:AVI} with step sizes \eqref{eq:tuning-1}.}
Let $\bm{x}^{*} = \left( \bm{v}^{*},\bm{v}^{*} \right) \in \R^{2n}$, where $\bm{v}^{*}$ is the value function of the optimal policy, and let $\bm{x}_{s} = \left( \bm{v}_{s},\bm{v}_{s-1}\right) \in \R^{2n}, s \geq 1.$
There exists a sequence of policies $(\hat{\pi}_{s})_{s \geq 0}$ such that
\begin{equation}\label{eq:LTV-AVI}
\bm{x}_{s} = \bm{x}^{*} +  \bm{B}_{\hat{\pi}_{s-1}} \cdot ... \cdot \bm{B}_{\hat{\pi}_{0}}(\bm{x}_{0}-\bm{x}^{*}), \forall \; s \geq 1.
\end{equation}
\end{proposition}

The dynamical system \eqref{eq:LTV-AVI} is said to be \textit{stable} if $(\bm{x}_{s})_{s \geq 0}$ converges to $\bm{0}$  for \textit{any} sequence of policies $(\hat{\pi}_{\ell})_{\ell \geq 0}$ and for any initial condition $\bm{x}_{0}$. Proposition \ref{prop:B-pi-IRU} highlights the key role played by the spectral radius of each matrix $\bm{B}_{\pi}$ and the \textit{joint spectral radius} of the set $\B = \{ \bm{B}_{\pi} \; | \; \pi \in \Pi \}$. The joint spectral radius of the set $\B$ is defined (among other equivalent definitions) as
\[ \rho(\B) = \limsup_{s \rightarrow \infty} \max \left\lbrace  \rho \left(\bm{B}_{\pi_{s}} \cdot ... \cdot \bm{B}_{\pi_{1}} \right)^{1/s} \; | \; \pi_{1},...,\pi_{s} \in \Pi \right\rbrace. \]
In particular, the sequence $(\bm{x}_{s})_{s \geq 0}$ is stable if and only if $\rho(\B) <1$ (Corollary 1.1 in \cite{jungers-book}). 

The authors in \cite{blondel-NP-hard} show that it is NP-hard to approximate the joint spectral radius for general set of matrices. The computation of $\rho(\mathcal{E})$ is known to be tractable for some special cases of set of matrices $\mathcal{E}$. In particular, if all the matrices in $\mathcal{E}$ are normal, or commonly triangularizable (Section 2.3.2 in \cite{jungers-book}), or row-wise nonnegative matrices (\cite{blondel-IRU}), it holds that $\rho(\mathcal{E})=\max_{\bm{E} \in \mathcal{E}} \rho(\bm{E}).$ This highlights the role of the spectral radius of each of the matrix $\bm{B}_{\pi}$. Therefore, one approach to study the convergence of Algorithm \ref{alg:AVI} would be to show a bound on the joint spectral radius of the set $\B$, the set of affine operators driving the dynamics of \eqref{eq:LTV-AVI}.
However, without any assumption, it may not hold that $\rho(\B) < 1$ (see our example in Section \ref{sec:disc-reversibility}). Even if we assume that the Markov chain associated with each policy $\pi$ is irreducible and reversible, we would only have $\rho(\bm{B}_{\pi})<1$, for all $\pi$ (see Lemma \ref{lem:radius}). However, this does not imply that $\rho(\B)<1$ (see Example 1.1 in~\cite{jungers-book}). While such a bound on the joint spectral radius is sufficient to prove convergence of Algorithm \ref{alg:AVI} for the Bellman operator $T$, it may not be a \textit{necessary} condition to prove the convergence of the sequence $(\bm{v}_{s})_{s \geq 0}$. In particular, there may be some more structure in the sequence of policies $(\hat{\pi}_{s})_{s \geq 0}$ of Proposition \ref{prop:B-pi-IRU}.  {In our numerical study (see Section \ref{sec:simu}), \ref{alg:AVI} always converges for a large class of MDP instances and exhibits significantly faster convergence times than classical approaches. We will present in Section \ref{sec:safe-AVI} an algorithm (Safe Accelerated Value Iteration, see Algorithm \ref{alg:SAVI}) which exhibits the strong empirical performances of \ref{alg:AVI} along with some convergence guarantees.}
\begin{remark} A similar analysis as for Proposition \ref{prop:B-pi-IRU} shows that the dynamics of the iterates of Momentum Value Iteration can also be modeled as a LTV dynamical system.  However, we notice in our numerical study (Section \ref{sec:simu}) that \ref{alg:MVI} may diverge on some of the instances that we considered. This is also the case in convex optimization, where \cite{heavyball-2015} give an example of an instance where momentum gradient descent may diverge when the objective function is only differentiable once. 
\end{remark}

\subsection{A family of hard MDP instances.}\label{sec:hard-instance} In light of the limitations of Algorithm \ref{alg:AVI} and \ref{alg:MVI}, we present a lower bound on the class of `first-order' iterative algorithms for MDP. We first recall the analogous results for convex optimization.
In optimization, a \textit{first-order algorithm} minimizing a differentiable, $\mu$-strongly convex, $L$-Lipschitz function $f: \R^{n} \rightarrow \R$ satisfies the following condition on the sequence of iterates $(\bm{x}_{s})_{s \geq 0}$: \[ \bm{x}_{s+1} \in\bm{x}_{0} + span \{ \nabla f (\bm{x}_{0}),...,\nabla f (\bm{x}_{s}) \}, \forall \; s \geq 0.\] \cite{nesterov-book} provides lower bounds on the convergence rate of any first-order algorithm on the class of smooth, convex functions and on the class of smooth, strongly-convex functions.  In particular, Nesterov's A-GD achieves the optimal rate of convergence over the class of smooth, convex functions, as well as over the class of smooth, strongly-convex functions. 
The proof relies on designing a hard instance.

Given our interpretation of $\left( \bm{I} -T \right)(\bm{v})$ as a gradient, in our MDP setting we consider \textit{first-order algorithm} as any iterative method where the sequence of iterates $(\bm{v}_{s})_{s \geq 0}$ satisfies 
\[ \bm{v}_{0}=\bm{0}, \bm{v}_{s+1} \in span \{\bm{v}_{0},...,\bm{v}_{s}, T(\bm{v}_{0}), ... T(\bm{v}_{s}) \}, s \geq 0.\]
We prove the following theorem, which present a lower bound on the convergence rate of first-order algorithms in all generality.
\begin{theorem}\label{th:hard-instance} There exists an MDP instance $(\X,\A,\bm{P},\bm{r},\bm{p}_{0},\lambda)$ such that for any sequence of iterates satisfying \[ \bm{v}_{0}=\bm{0}, \bm{v}_{s+1} \in span \{\bm{v}_{0},...,\bm{v}_{s}, T(\bm{v}_{0}), ... T(\bm{v}_{s}) \}, s \geq 0,\]
the following lower bound holds for any step $s \in \{1,...,n-1\}$: \[\| \bm{v}_{s} - \bm{v}^{*} \|_{\infty} \geq \dfrac{\lambda^{s}}{1+\lambda}.\]
\end{theorem}
\begin{proof}
We consider the following MDP. The discount factor is any $\lambda \in (0,1)$, there are $n $ states and one action. The rewards are such that $r_{1}=1$ and $r_{i}=0$ for all other states $i$. The state $1$ is absorbing and for $i \geq 2$, there is a deterministic transition from $i$ to $i-1$. 
\begin{figure}[H]
\centering 
\includegraphics[scale=0.25]{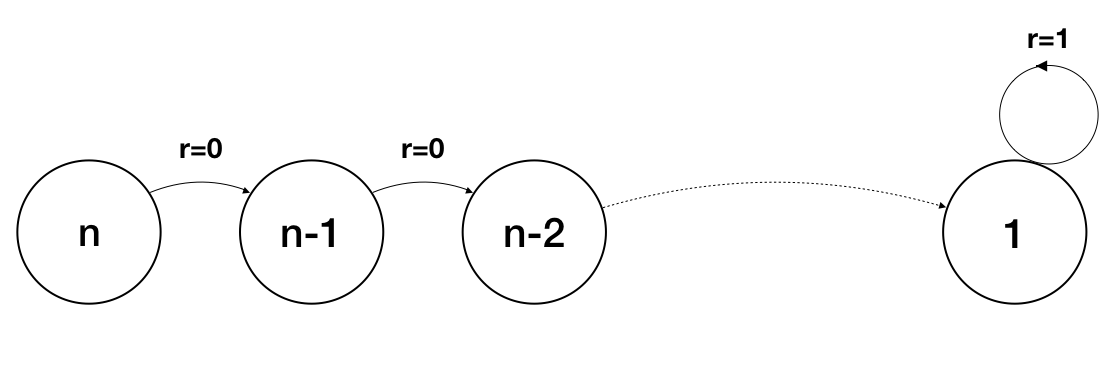}
\caption{An MDP where there are $n$ states and one action. The state $1$ is absorbing, $r_{1}=1$ and there is no reward in any other state. The arrows represent deterministic transitions.}
\label{fig:hard_instance}
\end{figure}

It is easy to see that the (optimal) value function is $v^{*}_{i}=\lambda^{i-1}/(1-\lambda), i \in \X.$
Let us consider a sequence of vectors $(\bm{v}_{s})_{s \geq 0}$ such that $\bm{v}_{0}=\bm{0}$ and
\[\bm{v}_{s} \in span \{ \bm{v}_{0}, ..., \bm{v}_{s-1}, T(\bm{v}_{0}), ... T(\bm{v}_{s-1}) \}, s \geq 0.\]
We prove by recursion that for all $s \geq 0, i \in \X$, we have $v_{s,i}=0$ if $i \geq s+1.$ This is true for $s=0$ because $\bm{v}_{0}=\bm{0}$. Let us assume that this is true for $v_{0},...,v_{s-1}$. Then given the definition of $T$ as in \eqref{eq:T_max-def} and the fact that $r_{i}=0$ for $i \geq 2$ in the MDP of Figure \ref{fig:hard_instance}, we have $T(\bm{v}_{t})_{i} = 0$ if $i \geq t+2$, for all $t \leq s-1$. Therefore, from 
$ \bm{v}_{s} \in span \{ \bm{v}_{0}, ..., \bm{v}_{s-1}, T(\bm{v}_{0}), ... T(\bm{v}_{s-1}) \}$ we see that $v_{s,i} = 0,$ for $i \geq s+1$, and we proved our recursion. 
The state $1$ is the only state where the decision-maker earns a reward; note that we essentially proved that any first-order algorithm takes at least $s-1$ steps to back-propagate the reward from state $1$ towards a state $1 \leq s \leq n$.

Now we have, for $1 \leq s \leq n-1$,
\begin{align}
\| \bm{v}_{s} - T(\bm{v}_{s}) \|_{\infty} & = \| \bm{v}_{s} - T(\bm{v}_{s}) - (\bm{v}^{*} - T(\bm{v}^{*})) \|_{\infty} \label{ineq:1} \\
& \geq (1-\lambda) \| \bm{v}_{s} - \bm{v}^{*} \|_{\infty} \label{ineq:2} \\
& \geq (1-\lambda) \max_{ 1 \leq i \leq n} |v_{s,i} - v^{*}_{i} | \nonumber \\
& \geq (1-\lambda) \max_{ s+1 \leq i \leq n} |v_{s,i} - v^{*}_{i} | \nonumber \\
& \geq (1-\lambda) \max_{ s+1 \leq i \leq n} | v^{*}_{i}| \label{ineq:3} \\
& \geq (1-\lambda) \max_{ s+1 \leq i \leq n} \dfrac{\lambda^{i-1}}{1-\lambda} \nonumber  \\
& \geq \lambda^{s}, \nonumber 
\end{align}
where \eqref{ineq:1} follows from $\bm{v}^{*} = T(\bm{v}^{*})$, \eqref{ineq:2} follows from \eqref{eq:mu-infty} and \eqref{ineq:2}, and \eqref{ineq:3} follows from $v_{s,i}=0$ for $i \geq s+1$.
We can conclude since triangle inequality gives $\| \bm{v}_{s} - \bm{v}^{*} \|_{\infty} \geq (1/(1+\lambda)) \cdot \| \bm{v}_{s} - T(\bm{v}_{s}) \|_{\infty}.$
\end{proof}
{
Therefore, \ref{alg:VI} matches the lower bound on the convergence rate of first-order algorithms. } Since \ref{alg:AVI} is also a first-order algorithm, Algorithm \ref{alg:AVI} may not perform better than \ref{alg:VI} in intermediate steps, \textit{before} the $n$-th iteration. 
We illustrate this for our hard MDP instance in Figure \ref{fig:hard_instance_errors}. We notice that \ref{alg:VI} and \ref{alg:R-VI} converge linearly with a rate close to $\lambda$ (as we prove in Proposition \ref{prop:VI-step size}). However, Algorithm \ref{alg:AVI} starts to converge faster than \ref{alg:VI} only after (at least) $n$ iterations. This suggests that if one wants to compute an $\epsilon$-optimal policy for a large $\epsilon$ in a high-dimensional MDP with an analogous structure as our hard MDP instance, \ref{alg:VI} might be a faster algorithm for this particular instance. 
\begin{figure}[H]
\centering
\includegraphics[scale=0.5]{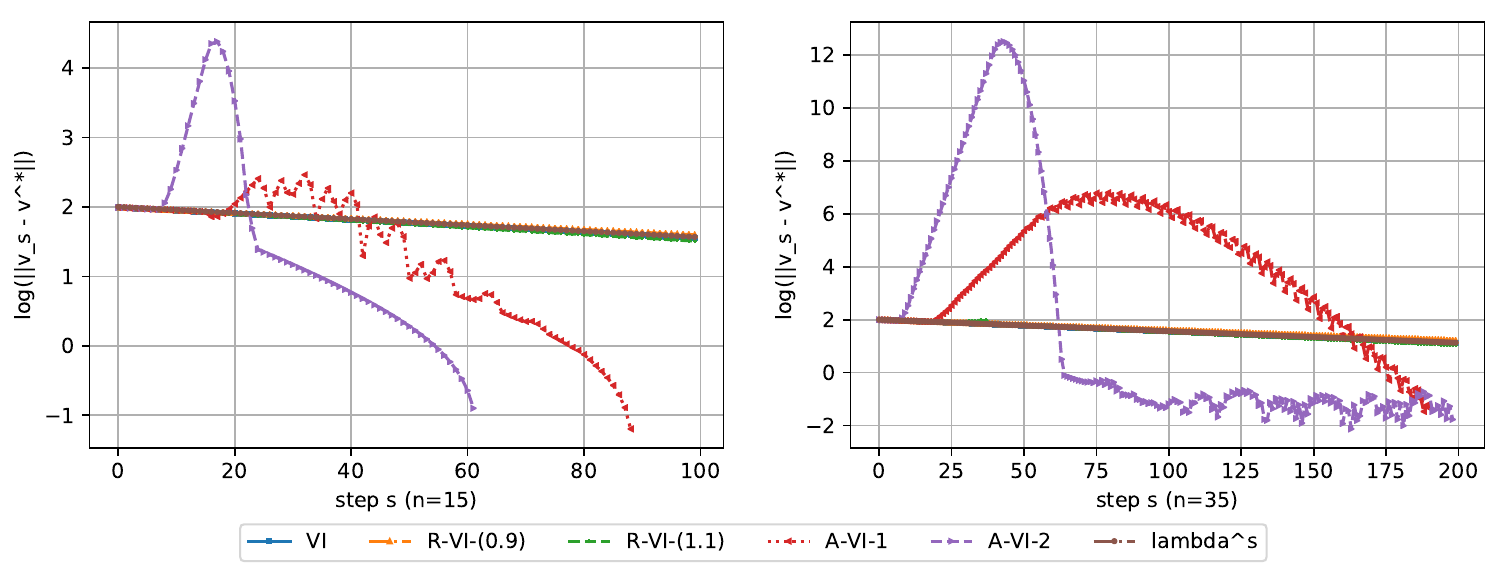}
\caption{We focus on the hard MDP instance of Figure \ref{fig:hard_instance}. We present the logarithm of the error $\| \bm{v}_{s} - \bm{v}^{*} \|_{\infty}$ for iterates of  \ref{alg:VI}, Algorithm \ref{alg:R-VI} for $\alpha=0.9$ and $\alpha=1.1$ (R-VI-(0.9) and R-VI-(1.1)), Algorithm \ref{alg:AVI} with the tuning \eqref{eq:tuning-1} (A-VI-1), Algorithm \ref{alg:AVI} with the more aggressive tuning $\alpha_{s}=\alpha=1,\gamma_{s}=\gamma=\left(1-\sqrt{1-\lambda}\right)^{2}/\lambda, \forall \; s \geq 1$, corresponding to $L=1, \mu=1-\lambda$ (A-VI-2). We also include $(\lambda^{s})_{s \geq 0}$, the rate of convergence of \ref{alg:VI}, for reference.}
\label{fig:hard_instance_errors}
\end{figure}
Finally, note that Algorithm \ref{alg:AVI} need not to be a \textit{descent algorithm}, i.e., it does not necessarily produce estimates that result in a monotonically decreasing objective function. The objective value might increase for a few periods, before significantly decreasing afterward (see Figure \ref{fig:hard_instance_errors}). This is analogous to \textit{oscillations} for accelerated gradient descent 
(see for instance Figure 1 in \cite{oscillation-2015} for a detailed study of the oscillation effect of A-GD).
{
\section{Balancing aggressive and safe algorithms: Safe Accelerated Value Iteration}\label{sec:safe-AVI}
In this section, we introduce an algorithm that balances the updates from \ref{alg:AVI} with updates from \ref{alg:VI}. The main motivation is that although accelerated algorithms may be too \textit{aggressive} for some MDP instances, we have seen in Section \ref{sec:analysis_AVC} that in the case of \ref{alg:AVC}, they can provide significant theoretical speedups for certain specific classes of MDP instances (e.g., reversible MDPs).  On the other hand,  \ref{alg:VI} is a \textit{safe} algorithm, that always converges regardless of the MDP instance,  even though the empirical convergence can be quite slow.  Therefore, we introduce Safe Accelerated Value Iteration (S-AVI,  see Algorithm \ref{alg:SAVI}), an algorithm which alternates between \ref{alg:AVI} updates and \ref{alg:VI} updates, depending of the progress made at the current iterate toward solving the equation $\bm{v} - T(\bm{v})=\bm{0}$ (see condition in Step \ref{alg:step:condition-SAVI} in Algorithm \ref{alg:SAVI}).
\begin{algorithm}[]
\caption{Safe Accelerated Value Iteration}\label{alg:SAVI}
\begin{algorithmic}[1]
\State\textbf{Initialize} $\lambda' \in [\lambda,1)$ and step sizes $\alpha,\gamma >0$.
\State\textbf{Initialize} $\bm{v}_{0}\in \R^{n},\bm{v}_{1}=T(\bm{v}_{0})$.
\For{each period $s \geq 1$}
 \State  Set
 \[
 \begin{cases}
    \bm{h}_{s}=\bm{v}_{s}+\gamma\cdot \left( \bm{v}_{s}-\bm{v}_{s-1} \right) \\
	\bm{v}_{s+1/2} \gets \bm{h}_{s}-\alpha \left( \bm{h}_{s}- T \left( \bm{h}_{s} \right) \right) \end{cases}\]
 
 \If {$\| \bm{v}_{s+1/2} - T(\bm{v}_{s+1/2}) \|_{\infty} \leq \lambda'^{s+1} \| \bm{v}_{0} - T(\bm{v}_{0})\|_{\infty}$} \label{alg:step:condition-SAVI}
 \State Set $\bm{v}_{s+1} = \bm{v}_{s+1/2}$.
\Else
\State Set $\bm{v}_{s+1} = T(\bm{v}_{s})$.
\EndIf  
\EndFor
\end{algorithmic}
\end{algorithm}
We give the convergence proof and the convergence rate of Algorithm \ref{alg:SAVI} in the next theorem.
\begin{theorem}\label{th:SAVI}
Consider an MDP instance and  $(\bm{v}_{s})_{s \geq 0}$ the sequence of iterates of Algorithm \ref{alg:SAVI} with scalar $\lambda' \in [\lambda,1)$ and any step sizes $\alpha,\gamma>0$.
Then $(\bm{v}_{s})_{s \geq 0}$ converges to $\bm{v}^{*}$.
Moreover,
 \[\bm{v}_{s} = \bm{v}^{*} + O\left(\lambda'^{s}\right), \forall \; s\geq 0.\]
\end{theorem} 
\begin{proof}
 Let $\left(\bm{v}_{s}\right)_{s \geq 0}$ be the output of Algorithm \ref{alg:SAVI}.  
We prove by induction that for any period $s \geq 0$, 
\[\| \bm{v}_{s} - T(\bm{v}_{s}) \|_{\infty} \leq \lambda'^{s} \| \bm{v}_{0} - T(\bm{v}_{0})\|_{\infty}.\]

This is straightforward for $s=0$. Also, this is true for $s=1$, because $\bm{v}_{1}=T(\bm{v}_{0})$ and $T$ is a contraction, so that
\[ \| \bm{v}_{1} - T(\bm{v}_{1}) \|_{\infty} = \|T( \bm{v}_{0}) - T(T(\bm{v}_{0})) \|_{\infty}  \leq \lambda \| \bm{v}_{0} - T(\bm{v}_{0})\|_{\infty} \leq \lambda' \| \bm{v}_{0} - T(\bm{v}_{0})\|_{\infty},\]
where the last inequality follows from $\lambda' \geq \lambda$.

Now suppose that for some $s \geq 1$ we have
\begin{equation}\label{eq:recursion}
\| \bm{v}_{s} - T(\bm{v}_{s}) \|_{\infty} \leq \lambda'^{s} \| \bm{v}_{0} - T(\bm{v}_{0})\|_{\infty}. 
\end{equation} We will show that 
\[ \| \bm{v}_{s+1} - T(\bm{v}_{s+1}) \|_{\infty} \leq \lambda'^{s+1} \| \bm{v}_{0} - T(\bm{v}_{0})\|_{\infty}.\]
This follows from the condition in Step \ref{alg:step:condition-SAVI}.
 Indeed, either the condition in Step \ref{alg:step:condition-SAVI} is not satisfied and we have $\bm{v}_{s+1} = T(\bm{v}_{s})$, which implies that
\[  \| \bm{v}_{s+1} - T(\bm{v}_{s+1}) \|_{\infty} \leq \lambda  \| \bm{v}_{s} - T(\bm{v}_{s}) \|_{\infty} \leq \lambda'  \| \bm{v}_{s} - T(\bm{v}_{s}) \|_{\infty} \leq \lambda'^{s+1} \| \bm{v}_{0} - T(\bm{v}_{0})\|_{\infty},\]
where the last inequality follows from the recursion hypothesis \eqref{eq:recursion}.

Otherwise, the condition in Step \ref{alg:step:condition-SAVI} is satisfied and we have \[ \| \bm{v}_{s+1} - T(\bm{v}_{s+1}) \|_{\infty} \leq \lambda'^{s+1}  \| \bm{v}_{0} - T(\bm{v}_{0}) \|_{\infty}.\]
Overall, we have proved by recursion that for all $s \in \N$, we have
$\| \bm{v}_{s} - T(\bm{v}_{s}) \|_{\infty} \leq \lambda'^{s}  \| \bm{v}_{0} - T(\bm{v}_{0}) \|_{\infty}.$

We now show that this is enough to prove Theorem \ref{th:SAVI}. We have
\[ \| \bm{v}_{s} - \bm{v}^{*} \|_{\infty} =  \| \bm{v}_{s} + T(\bm{v}_{s}) - T(\bm{v}_{s}) - \bm{v}^{*} \|_{\infty} \leq \lambda \| \bm{v}_{s} - \bm{v}^{*} \|_{\infty} +  \| \bm{v}_{s} + T(\bm{v}_{s})\|_{\infty}, \]
which in turn implies that
$ \| \bm{v}_{s} - \bm{v}^{*} \|_{\infty} \leq (1-\lambda)^{-1} \| \bm{v}_{s} - T(\bm{v}_{s} \|_{\infty}$. But we just proved that \[ \| \bm{v}_{s} - T(\bm{v}_{s}) \|_{\infty} = O \left( \lambda'^{s} \right).\]
Therefore,  $\| \bm{v}_{s} - \bm{v}^{*} \|_{\infty} = O \left( \lambda'^{s} \right)$, i.e., $\left( \bm{v}_{s} \right)_{s \geq 0}$ is a sequence of vectors converging to $\bm{v}^{*}$ and $\bm{v}_{s} = \bm{v}^{*} + O \left( \lambda'^{s} \right).$
\end{proof}
Several remarks are in order.
\begin{remark}[Rate of convergence and choice of $\lambda'$] S-AVI converges to $\bm{v}^{*}$ at a linear rate of $\lambda'$ for $\lambda' \in [\lambda,1)$.  This means that the theoretical convergence rate of Algorithm \ref{alg:SAVI} cannot be better than the rate of convergence of \ref{alg:VI}. However,  we have proved in Theorem \ref{th:hard-instance} that in all generality, no algorithm can achieve a faster convergence rate than \ref{alg:VI}, i.e., Algorithm \ref{alg:SAVI} with $\lambda'=\lambda$ matches the theoretical lower bound on the convergence rates of any first-order algorithm for MDPs.
Additionally, one can control the number of AVI steps and VI steps in Algorithm \ref{alg:SAVI} with the parameter $\lambda' \in [\lambda,1)$. Intuitively,  for $\lambda'$ closer to $1$, the condition in Step \ref{alg:step:condition-SAVI} is weaker and more AVI steps are taken than for $\lambda'$ close to $\lambda$.  Our analysis from Section \ref{sec:analysis_MAVC} shows that aggressive updates have the potential to benefits the empirical convergence of the algorithm,  and we will see in the numerical experiments that using $\lambda' = (1+\lambda)/2$ results in significant speedups compared to classical algorithms.
\end{remark}
\begin{remark}[Condition to activate aggressive updates]
Other conditions than the one presented in Step \ref{alg:step:condition-SAVI} of Algorithm \ref{alg:SAVI} could be used and would lead to the same convergence rate for Algorithm \ref{alg:SAVI}. In particular, one could use
\begin{equation}\label{eq:alt-condition-1}
\| \bm{v}_{s+1/2} - T(\bm{v}_{s+1/2}) \|_{\infty}  \leq \lambda' \| \bm{v}_{s} - T(\bm{v}_{s})\|_{\infty}.
\end{equation}
 This condition would also work because it holds when $\bm{v}_{s+1/2} = T(\bm{v}_{s})$, i.e.,  even when the algorithm takes a VI step.
Note that \eqref{eq:alt-condition-1} enforces progress at every period, i.e., from $\bm{v}_{s}$ to $\bm{v}_{s+1}$ at every period $s$.  However, we know that \ref{alg:AVI} may not be a descent algorithm, and this is the main reason why the condition chosen for Step \ref{alg:step:condition-SAVI},  $\| \bm{v}_{s+1/2} - T(\bm{v}_{s+1/2}) \|_{\infty} \leq \lambda'^{s+1} \| \bm{v}_{0} - T(\bm{v}_{0})\|_{\infty}$, only imposes progress of $\bm{v}_{s+1}$ compared to $\bm{v}_{0}$ and not compared to $\bm{v}_{s}$. Empirically, we have observed that S-AVI as presented in Algorithm \ref{alg:SAVI} performs better than with condition \eqref{eq:alt-condition-1} and almost always take A-VI updates instead of VI updates.  We present the details of our numerical experiment in our next section.
\end{remark}
\begin{remark}[Generalization: Safe Fixed-Point Iteration]
The idea of interweaving aggressive and safe steps is very general and has been used for instance in \cite{ref-a} in the case of \textit{Anderson acceleration}.  Algorithm \ref{alg:SAVI} can be considered with the operator $T_{\pi}$ instead of $T$,  or even for any contracting operator of rate $\lambda$.  Similarly, for the intermediary point $\bm{v}_{s+1/2}$,  one could consider other updates than \ref{alg:AVI}, e.g.,  one could consider \ref{alg:MVI} or even other aggressive methods.  Theorem \ref{alg:SAVI} would still hold, as the proof is independent of the choice of $\bm{v}_{s+1/2}$ and solely relies on the operator being a contraction of rate $\lambda$. We choose to focus on Safe A-VI in this section because we noticed empirically that it performs better than Safe Momentum Value Iteration (see details in the next section).
\end{remark}
}
\section{Numerical study.}\label{sec:simu}
{
We now present our numerical experiments. We compare S-AVI with four classical approaches, including Value Iteration, Policy Iteration~\citep{ye-2011}, Gauss-Seidel Value Iteration~\citep{Puterman} and Anderson Value Iteration~\citep{ref-c},  for solving both structured and random MDP instances.  These are state-of-the-art approaches that range from first-order methods for MDPs (VI, Gauss-Seidel VI) to second-order methods (as PI is analogous to Newton method) and quasi-Newton methods (Anderson VI is related to Anderson acceleration in convex optimization). 
The details about the MDP instances and our implementation of all algorithms can be found in Appendix \ref{app:details-numerical-exp}.}

{
\vspace{1mm}
\noindent \textbf{Experimental setup.}
All figures in this section present the logarithm of the running times (in second) of the algorithms before returning an $\epsilon$-optimal policy.
The simulations are performed on a laptop with 2.2 GHz Intel Core i7 and 8 GB of RAM.  We stop the algorithms when they are guaranteed to return an $\epsilon$-optimal solution with the condition $\| \bm{v}_{s} - T(\bm{v}_{s}) \|_{\infty} \leq \epsilon \cdot (1-\lambda)$. All algorithms are initialized at $\bm{v}_{0}=\bm{0}$, and when applicable, $\bm{v}_{1}=T(\bm{v}_{0})$. We choose $\epsilon=0.1$ in our simulations. For the simulations over random MDP instances (Garnet MDPs, Figure \ref{fig:garnet}), the performances are averaged over 10 randomly generated instances.  For S-AVI we choose $\lambda' = (1+\lambda)/2 \in (\lambda,1)$, which guarantees the convergence the convergence of the algorithm at a rate at least $\lambda'$. }

{\noindent \textbf{MDP instances.} 
We compare the performances of various algorithms on two structured MDP instances (\textit{forest} and \textit{healthcare} MDPs) and some random MDP instances (\textit{Garnet} MDPs).  In the first instance inspired from real-world application, the \textit{forest MDP}~\citep{pymdp}, a forest grows over the years and the decision maker balances the revenues associated with selling cut wood with the increasing risk of wildfire.  This is inspired from an application of dynamic programming to optimal fire management~\citep{possingham1997application}. The second structured instance, the \textit{healthcare MDP}, is a simplification of the Markov chains for modeling the dynamic evolution of the health of the patients as presented in \cite{Goh} and \cite{grand2020robust}. For our experiments with randomly generated MDPs, we use the class  of Generalized Average Reward Non-stationary Environment Test-bench (Garnet MDPs), introduced in  \cite{garnet}. 
 Garnet MDPs are a class of abstract but representative  MDPs. In particular, we can control the connectivity of the underlying Markov chain with a \textit{branching} factor, $n_{\sf branch}$, which represents the proportion of next states available at every state-action pair $(i,a)$. Garnet MDPs serve as a test-bench for RL algorithms \citep{garnet-1,garnet-2,garnet-3}.  We present the details about each MDP instance in Appendix \ref{app:details-MDP-instances}. In our experiment we choose $n_{\sf branch} = 80 \%$ and $A=50$, to compare the algorithms on instances denser than the forest and healthcare MDPs.
}
\begin{figure}[h!]
\includegraphics[width=0.9\linewidth]{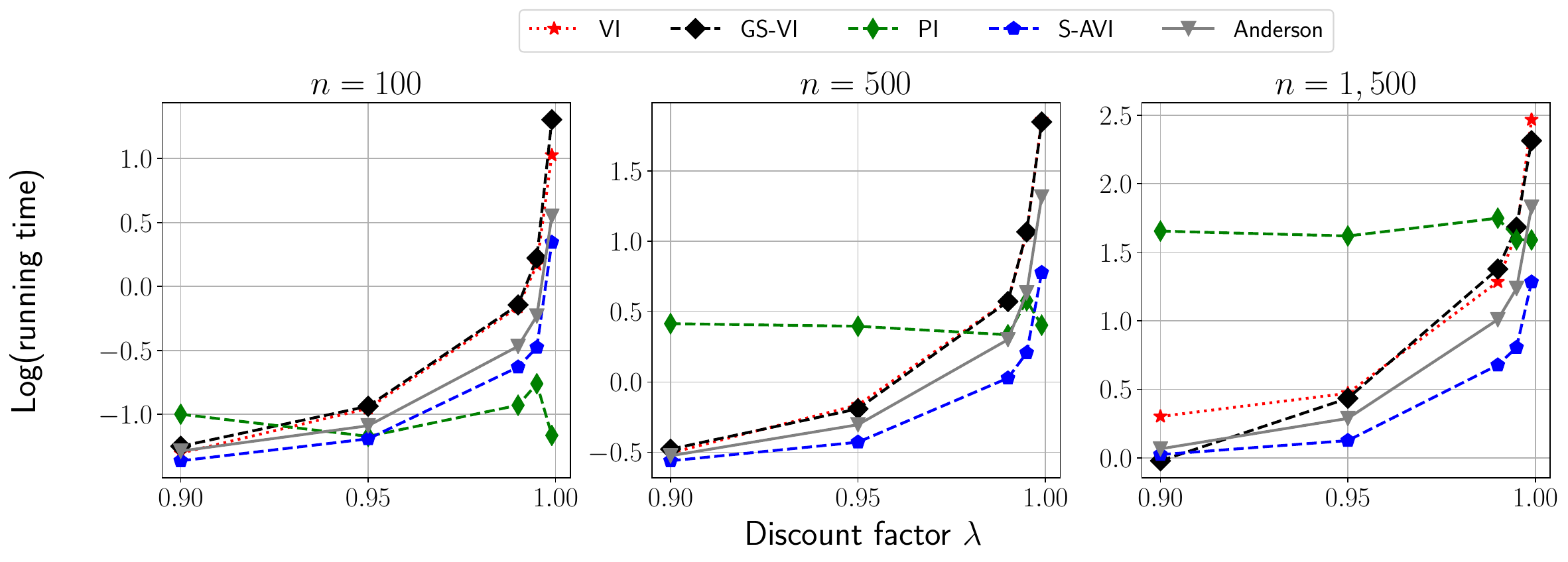}
\caption{Logarithm of the running times of various algorithms for solving the forest MDP instance.}
\label{fig:forest}
\end{figure}
\begin{figure}[h!]
\includegraphics[width=0.9\linewidth]{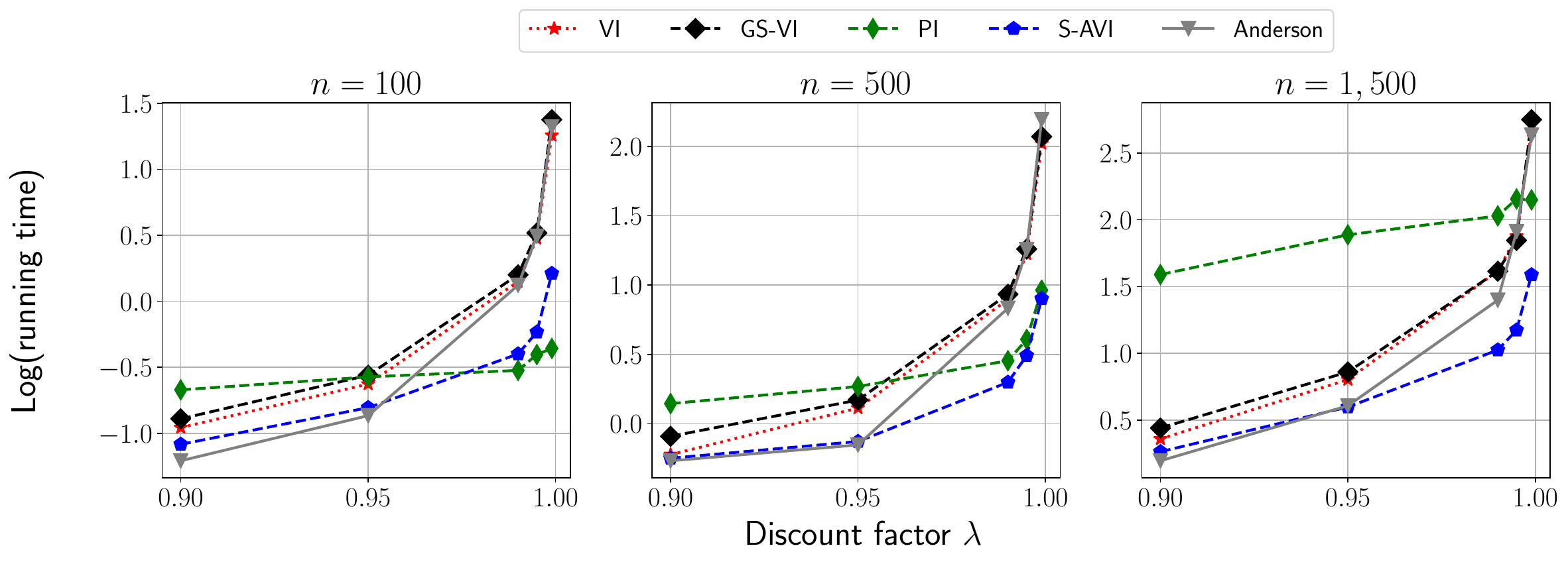}
\caption{Logarithm of the running times of various algorithms for solving the healthcare MDP instance.}
\label{fig:healthcare}
\end{figure}
\begin{figure}[h!]
\includegraphics[width=0.9\linewidth]{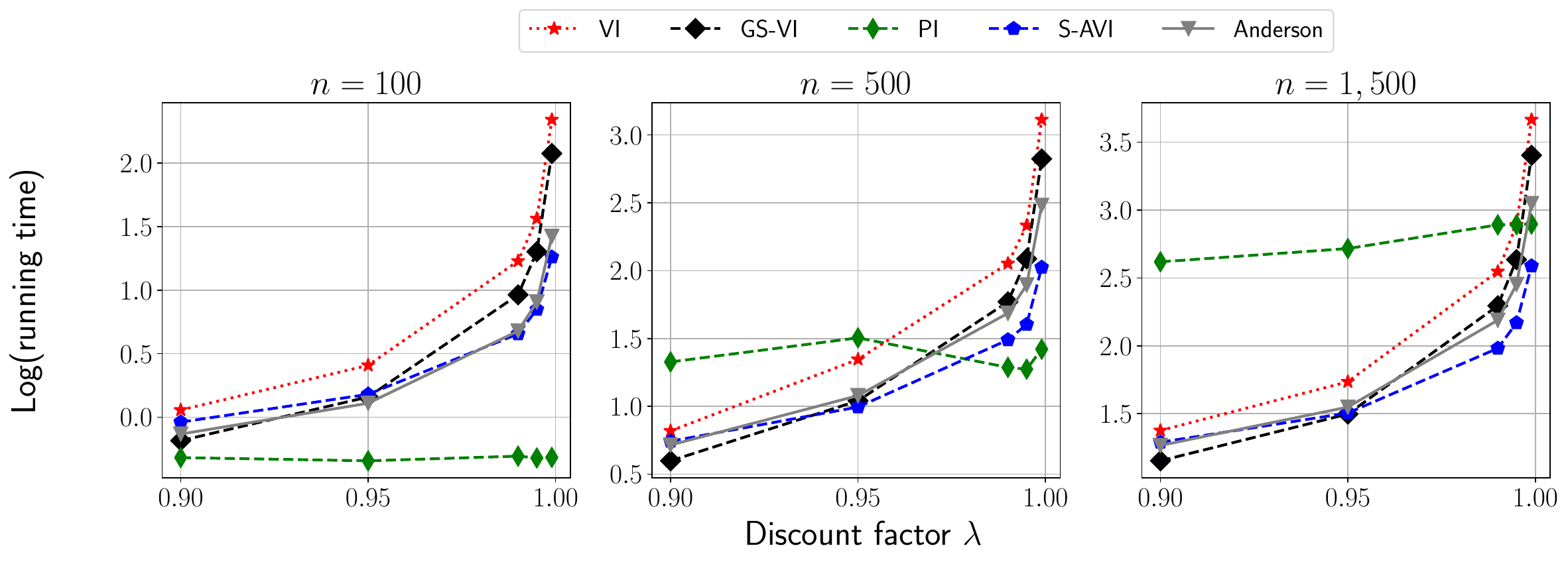}
\caption{Logarithm of the running times of various algorithms for solving Garnet MDP instances.}
\label{fig:garnet}
\end{figure}

{
\vspace{1mm}
\noindent \textbf{Numerical results.} We present in Figures \ref{fig:forest}-\ref{fig:garnet} our numerical results.  The two best algorithms are always either PI either S-AVI.  Policy Iteration may outperform S-AVI when  the number of states is small and the discount factor is high (e..g, $n=100$ for all figures), whereas for large number of states ($n=1,500$) S-AVI typically outperforms all algorithms.  Policy Iteration becomes slower for large state spaces because it requires to invert a matrix at each iteration to compute the value function of the current policy, see details in Appendix \ref{app:details-other-algorithms}.
We also note that GS-VI improves upon VI only for dense instances (here, Garnet MDPs) and not for the structured more sparse instances (forest and healthcare MDPs). Anderson VI may  be competitive with S-AVI for small values of $\lambda$ (typically $\lambda =0.90$ or $\lambda=0.95$) but its running times increase substantially for $\lambda \geq 0.99$.
 Finally,  we note that VI is almost always the slowest algorithm and that on all instances, the improvement of S-AVI compared to VI is at least one order of magnitude when $\lambda$ approaches $1$.
}

 {
\vspace{1mm}
\noindent \textbf{Relation between A-VI and S-AVI.} In our simulations, we notice that S-AVI takes an aggressive step (i.e.,  the condition in Step \ref{alg:step:condition-SAVI} of S-AVI is satisfied) for more than $99 \%$ of the steps of S-AVI across all MDP instances, at any values of $\lambda$ and $n$ that we have tested.  Additionally, even though we do not have a theoretical convergence result for A-VI, we note that it converges on all our test instances, for both random and structured instances.  The performances of S-AVI are almost exactly similar as the ones of A-VI; A-VI is slightly faster because it does not check if the condition in Step \ref{alg:step:condition-SAVI} of S-AVI is verified.  Since A-VI and S-AVI have analogous performances,  we only show the performances of S-AVI in Figures \ref{fig:forest}-\ref{fig:garnet}.} 

{
\vspace{1mm}
\noindent \textbf{Performances of Safe M-VI and divergence of \ref{alg:MVI}.} 
In our simulations we also considered Safe M-VI, a version of Algorithm \ref{alg:SAVI} where the intermediate point $\bm{v}_{s+1/2}$ is generated by \ref{alg:MVI} (instead of \ref{alg:AVI}).  We found that S-MVI does not outperform S-AVI, and may even require as much time as \ref{alg:VI} before convergence.  Empirically, it only takes M-VI steps approximately $70 \%$ of the time, compared to S-AVI that takes A-VI steps more than $99 \%$ of the time in our experiments.  We also observed that Algorithm \ref{alg:MVI} might diverge for the natural choice of step sizes for twice differentiable functions, i.e. for tuning  \eqref{eq:tuning-MVC-1}. This is analogous to the situation in convex optimization, where
 the optimal choice of parameters for $f$ a twice differentiable function might lead to a non-converging sequence for a function with a piece-wise linear gradient, i.e. when $f$ is only differentiable once (\citet{heavyball-2015}).  Note that $ \bm{v} \mapsto (\bm{I} - T)(\bm{v})$ is only a piece-wise linear function and is not differentiable everywhere.  In contrast,  accelerated gradient descent is guaranteed to converge even when $f$ is only differentiable once.}
\section{Conclusion.}
In this paper we present a fundamental connection between gradient descent and value iteration and build upon this analogy and ideas from Nesterov's acceleration to present \ref{alg:AVI}, an accelerated value iteration algorithm for MDP.  {We obtain insights on the theoretical convergence rate of accelerated algorithms in the case of affine operators on reversible MDP instances,  where the theoretical speedup is analogous to the situation in convex optimization.  We show that in all generality,  \ref{alg:AVI} may fail to converge,  prove a theoretical lower bound on the convergence rate of any first-order method for MDPs and show that \ref{alg:VI} matches this lower bound.  Based on this, we introduce Safe Accelerated Value Iteration (S-AVI), an algorithm that interweaves the aggressive \ref{alg:AVI} updates with safe \ref{alg:VI} updates. S-AVI exhibits the strong empirical performances of \ref{alg:AVI} combined with the theoretical convergence guarantees of \ref{alg:VI}, matching the lower bound on theoretical convergence rates. Interesting next steps include designing novel quasi-Newton methods for MDPs and considering other non-expansive operators than the Bellman operator.}
\small
\bibliographystyle{plainnat}
\bibliography{ref}
\normalsize
\appendix
{
\section{When is the Bellman operator the gradient of a function?}\label{app:non-existence-of-f}
In this appendix we characterize the MDP instances for which there exists a function $f: \R^{n} \rightarrow \R$ such that $\nabla f (\bm{v}) = \left(\bm{I} - T\right)(\bm{v})$. 

First, note that if $\nabla f (\bm{v}) =T(\bm{v})$, then for $g:\R^{n} \rightarrow \R, g(\bm{v}) = (1/2)\bm{v}^{\top}\bm{v} - f$, we have $\nabla g (\bm{v}) = \left(\bm{I} - T\right)(\bm{v})$.  Therefore, we will focus on whether there exists $f: \R^{n} \rightarrow \R$ for which $\nabla f (\bm{v}) =T(\bm{v})$.

\vspace{2mm}
\noindent
\textbf{Tools from differential geometry.} We will need several results from differential geometry.  Because the main focus of this paper is solving MDPs,  we state the next definitions and lemmas for our specific setting. We start with the following definition of a {\it $1$-form}, which can be found page 130 in \cite{lee2013smooth}. 
\begin{definition}
A \textit{$1$-form} of $\R^{n}$ is a map $\omega: \R^{n} \rightarrow \R^{n}$.
\end{definition}
$1$-forms are sometimes referred to as {\it vector fields}, because they map each vector $\bm{v}$ from $\R^{n}$ to another vector $\omega(\bm{v})$ of $\R^{n}$.
Note that both $T$ and $T_{\pi}$ are $1$-forms.   We then need the definitions of \textit{closed} and \textit{exact} $1$-forms, which can be found page 310 in \cite{lee2013smooth}. We also need the definition of the \textit{exterior derivative} of a $1$-form, which generalizes the differential of a function (see page 305 in \cite{lee2013smooth}).
\begin{definition}
\begin{itemize}
\item {\it Exterior derivative.} The exterior derivative $d \omega$ of a differentiable $1$-form $\omega: \R^{n} \rightarrow \R^{n}$ is a map $d \omega : \R^{n} \rightarrow \R^{n} \times \R^{n}$ such that for any $\bm{v} \in \R^{n}$,
\begin{equation}\label{eq:definition-exterior-derivative}
\left( d \omega (\bm{v}) \right)_{ij} = \frac{\partial \omega_{j}}{\partial v_{i}}(\bm{v}) -  \frac{\partial \omega_{i}}{\partial v_{j}}(\bm{v}), \forall \; (i,j) \in \{1,...,n\}^{2}.
\end{equation}
\item {\it Closed form.} A differentiable $1$-form $\omega: \R^{n} \rightarrow \R^{n}$ is closed if $d \omega = 0$.
\item {\it Exact form.} A differentiable $1$-form $\omega: \R^{n} \rightarrow \R^{n}$ is exact if there exists $\alpha : \R^{n} \rightarrow \R$, for which $d \alpha = \omega$. In this case, $\alpha$ is called a {\it primitive} of $\omega$.
\end{itemize}
\end{definition}
The next two lemmas show the equivalence of exact forms and closed forms when the domain is simply connected and the $1$-forms are differentiable.
\begin{lemma}[Lemma 6.26, \cite{lee2013smooth}; Schwarz's Lemma.]\label{lem:schwartz}
Any exact differentiable $1$-form is closed.
\end{lemma}
\begin{lemma}[Theorem 15.17, \cite{lee2013smooth}: Poincare's Lemma.]\label{lem:poincare}
Let $\omega: \D \rightarrow \R^{n}$ a closed differentiable $1$-form defined on an open set $\D \subset \R^{n}$.  If $\D$ is simply connected, then $\omega$ is exact.
\end{lemma}

Lemma \ref{lem:schwartz} simply follows from the fact that the Hessian of a twice-differentiable function $\alpha: \R^{n} \rightarrow \R$ is always a symmetric matrix. The proof of Lemma \ref{lem:poincare} is constructive: one can define the function $\alpha: \R^{n} \rightarrow \R$ for which $d \alpha = \omega$ as the (multi-dimensional) integral of $\omega$ along paths in the set $\D$, as long as the set $\D$ is simply connected.

We are now ready to characterize the MDP instances where the operator $T$ has a primitive.  To provide a better intuition of our results, we start by answering this question for $T_{\pi}$, defined in Equation \eqref{eq:T_pi}.

\vspace{2mm}
\noindent
\textbf{The case of $T_{\pi}$.}
Let $\pi$ a policy.  Recall that we have defined $\bm{r}_{\pi} \in \R^{n}, \bm{L}_{\pi} \in \R^{n \times n}$ as
\[ r_{\pi,i} = \sum_{a \in \A} \pi_{ia}r_{ia}, L_{\pi,ij} = \sum_{a \in \A} \pi_{ia}P_{iaj}, \forall \; (i,j) \in \{1,...,n\}^{2}.\]
With these notations, we can write $T_{\pi}: \bm{v} \mapsto \bm{r}_{\pi} + \lambda \bm{L}_{\pi}\bm{v}$. Because the operator $T_{\pi}$ is affine, it is always differentiable. By definition, it is a $1$-form, and it is defined on $\R^{n}$, which is simply connected.  Therefore,  from Lemma \ref{lem:schwartz} and Lemma \ref{lem:poincare} the existence of a function $f_{\pi}: \R^{n} \rightarrow \R$ such that $\nabla f_{\pi} = T_{\pi}$ is exactly equivalent to 
\begin{equation}\label{eq:exterior-deriv-T-pi}
 \frac{\partial T_{\pi}(\bm{v})_{j}}{\partial v_{i}}(\bm{v}) =  \frac{\partial T_{\pi}(\bm{v})_{i}}{\partial v_{j}}(\bm{v}), \forall \; (i,j) \in \{1,...,n\}^{2}.
\end{equation}
Equation \eqref{eq:exterior-deriv-T-pi} is equivalent to
\begin{equation}\label{eq:symmetric-instance-T-pi}
L_{\pi,ij} = L_{\pi,ji},  \forall \; (i,j) \in \{1,...,n\}^{2},
\end{equation}
i.e.,  we have proved the following proposition.
\begin{proposition} Let $\pi$ a policy.
There exists a function $f_{\pi}: \R^{n} \rightarrow \R$ such that $\nabla f_{\pi} = T_{\pi}$ if and only if $\bm{L}_{\pi}$ is symmetric, where $\bm{L}_{\pi}$ is the transition matrix of the Markov chain induced by $\pi$ on the set of states $\{1,...,n\}$.
\end{proposition}
Note that if $\bm{L}_{\pi}$ is symmetric, we can even write the closed-form expression of $f_{\pi}$ as
\[ f_{\pi}(\bm{v}) = c_{\pi} + \bm{v}^{\top}\bm{r}_{\pi} + \frac{\lambda}{2} \bm{v}^{\top}\bm{L} \bm{v},\]
where $c_{\pi} \in \R$ is any constant.

\vspace{2mm}
\noindent
\textbf{The case of $T$.} For the Bellman operator $T$, the question of the existence of a primitive is slightly more complex, because 1) $T$ is not differentiable, and 2) $T$ is not affine. However,  $T$ is piece-wise affine, as the maximum of some affine functions.  Let us introduce the following notations.
\begin{enumerate}
\item For each policy $\pi \in \Pi$,  let $\V_{\pi}$ the region of $\R^{n}$ where $T$ and $T_{\pi}$ coincides:
\[\V_{\pi} = \{ \bm{v} \in \R^{n} \; | \; T_{\pi}(\bm{v}) = T(\bm{v}) \}.\]
\item Let $\Pi^{\circ}$ the set of deterministic policies for which $\V_{\pi}$ has non-empty interior:
\[ \Pi^{\circ} = \{ \pi \in \Pi \; | \;  \pi \text{ is deterministic and } \V_{\pi} \text{ has non empty interior } \}.\]
\end{enumerate}
Note that the set $\Pi^{\circ}$ is always finite, because the set of deterministic policies is finite.
By definition of $T$ and $T_{\pi}$,  $\V_{\pi}$ is always a polytope. Therefore, $\V_{\pi}$ is always simply connected. Additionally, $\bigcup_{\pi \in \Pi^{\circ}} \V_{\pi}$ is a covering of $\R^{n}$.  To apply Lemma \ref{lem:schwartz} and Lemma \ref{lem:poincare}, we need a domain that is simply connected {\it and} open.  We first obtain the following proposition, where the primitives are defined on the interiors of the sets $\V_{\pi}$.
\begin{proposition}\label{prop:T-gradient-field-1}
Assume that for all policy $\pi$ in $\Pi^{\circ}$, the matrix $\bm{L}_{\pi}$ is symmetric. Then there exists a finite set of differentiable functions $\F_{\pi}=\{ f_{\pi}: int(\V_{\pi}) \rightarrow \R \; | \; \pi \in \Pi^{\circ} \}$ such that
\[ \forall \; \pi \in \Pi^{\circ},  \forall \; \bm{v} \in int(\V_{\pi}),  \nabla f_{\pi}(\bm{v}) = T(\bm{v}).\]
\end{proposition}
By using an argument of extension by continuity (e.g., Kirszbraun's theorem~\citep{valentine1945lipschitz}),  we can extend each function $f_{\pi}: int(\V_{\pi}) \rightarrow \R$ to functions $\tilde{f}_{\pi}: \V^{\pi} \rightarrow \R$.  Therefore,  we obtain the following proposition, which characterizes the MDP instances for which $T$ has a primitive.
\begin{proposition}\label{prop:T-gradient-field-2}
Assume that for all policy $\pi$ in $\Pi^{\circ}$, the matrix $\bm{L}_{\pi}$ is symmetric. Then there exists a finite set of differentiable functions $\tilde{\F}_{\pi}=\{ \tilde{f}_{\pi}: \V_{\pi} \rightarrow \R \; | \; \pi \in \Pi^{\circ} \}$ such that
\[ \forall \; \pi \in \Pi^{o},  \forall \; \bm{v} \in \V_{\pi},  \nabla \tilde{f}_{\pi}(\bm{v}) = T(\bm{v}).\]
\end{proposition}
}

\section{Proofs of Section \ref{sec:analysis-MVC}}\label{app:proof-section-MVC}
In this section we give three lemmas that we use in our proof of Theorem \ref{th:MVC-1}. 
 We first analyze the structure of the iterations in Algorithm \ref{alg:MVC}, with constants \eqref{eq:tuning-MVC-1}. Let $\bm{I} \in \R^{n \times n}$ be the identity matrix and $ \bm{B}_{\pi}  \in \R^{2n \times 2n}$ defined as
\begin{equation*}
 \bm{B}_{\pi} = \left[ {\begin{array}{cc}
\lambda (\beta+1) \bm{L}_{\pi} & - \beta \bm{I} \\
    \bm{I}  & \bm{0}
  \end{array} } \right].
\end{equation*}
\begin{lemma}\label{lem:LTV-MVC} 
Let $\bm{x}^{\pi} = (\bm{v}^{\pi},\bm{v}^{\pi}) \in \R^{2 n},\bm{x}_{s} = (\bm{v}_{s},\bm{v}_{s-1}) \in \R^{2 n}$. Then $$\bm{x}_{s} = \bm{x}^{\pi}+ \bm{B}_{\pi}^{s-1}\left(\bm{x}_{0}-\bm{x}^{\pi}\right), \forall \; s \geq 1.$$
\end{lemma}
\begin{proof}[Proof of Lemma \ref{lem:LTV-MVC}.]
Let $\alpha,\beta \in \R$. Expanding the recursion in Algorithm \ref{alg:MVC}, we have, for any $s \geq 1, \bm{v}_{s+1} =\left( (1-\alpha+\beta) \bm{I} + \alpha \lambda \bm{L}_{\pi} \right) \bm{v}_{s} - \beta \bm{v}_{s-1} + \alpha \bm{r}_{\pi}.$
%\begin{align*}
%\bm{v}_{s+1} & = \bm{v}_{s} - \alpha (\bm{I} - \bm{T}_{\pi})(\bm{v}_{s}) + \beta (\bm{v}_{s} - \bm{v}_{s-1}) \\
%& = \bm{v}_{s} - \alpha (\bm{v}_{s} - \bm{r}_{\pi} - \lambda \bm{L}_{\pi}\bm{v}_{s}) + \beta (\bm{v}_{s} - \bm{v}_{s-1}) \\
%& = \left( (1-\alpha+\beta) \bm{I} + \alpha \lambda \bm{L}_{\pi} \right) \bm{v}_{s} - \beta \bm{v}_{s-1} + \alpha \bm{r}_{\pi}.
%\end{align*}
For $\alpha,\beta$ tuned as in \eqref{eq:tuning-MVC-1}, we have $\alpha = 1 + \beta$.
Therefore, $\bm{x}_{s}=\bm{B}_{\pi}\bm{x}_{s-1} + \bm{b}_{\pi}$,  for $\bm{x}_{s} = (\bm{v}_{s},\bm{v}_{s-1})$, for
\[
 \bm{B}_{\pi} = \left[ {\begin{array}{cc}
\lambda (\beta+1) \bm{L}_{\pi} & -\beta\bm{I} \\
    \bm{I}  & \bm{0}
  \end{array} } \right],  \]
  \[\bm{b}_{\pi} =\alpha \left[ {\begin{array}{c}
\bm{r}_{\pi} \\
\bm{0}
  \end{array} } \right]. \]
The value function $\bm{v}^{\pi}$ is the unique solution to the Bellman equation $\bm{v}^{\pi} = \bm{r}_{\pi} + \lambda \bm{L}_{\pi} \bm{v}^{\pi}$. This readily implies that $\bm{x}^{\pi} = (\bm{v}_{\pi},\bm{v}_{\pi})$ satisfies $\bm{x}^{\pi} = \bm{B}_{\pi} \bm{x}^{\pi} + \bm{b}_{\pi}$. Therefore,
\begin{align*}
\bm{x}_{s}& =\bm{B}_{\pi}\bm{x}_{s-1} + \bm{b}_{\pi}  = \bm{B}_{\pi}\bm{x}^{\pi}+ \bm{B}_{\pi}(\bm{x}_{s-1}-\bm{x}^{\pi}) + \bm{b}_{\pi}  = \bm{x}^{\pi} +\bm{B}_{\pi}(\bm{x}_{s-1}-\bm{x}^{\pi}),
\end{align*}
from which we can conclude that
$ \bm{x}_{s} = \bm{x}^{\pi} + \bm{B}_{\pi}^{s}(\bm{x}_{0}-\bm{x}^{\pi}).$
\end{proof}

The next lemma bounds the spectral radius of the matrix $\bm{B}_{\pi}$, for Algorithm \ref{alg:MVC} with constants \eqref{eq:tuning-MVC-1}.
\begin{lemma}\label{lem:radius-tuning-MVC-1} Consider a policy $\pi$ which defines an irreducible, reversible Markov chain on the set of states $\X$. Consider constant step sizes $\alpha,\beta$ as in \eqref{eq:tuning-MVC-1}. Then we have
\[\rho(\bm{B}_{\pi}) = \dfrac{1-\sqrt{\kappa}}{1+\sqrt{\kappa}}.\]
\end{lemma}
\begin{proof}[Proof of Lemma \ref{lem:radius-tuning-MVC-1}.]
For $\alpha = 2/(1+\sqrt{1-\lambda^{2}}), \beta =(1-\sqrt{1-\lambda^{2}})/(1+\sqrt{1-\lambda^{2}})$, the matrix $\bm{B}_{\pi}$ becomes
\[ \bm{B}_{\pi} = \left[ {\begin{array}{cc} \dfrac{2 \lambda}{1+\sqrt{1-\lambda^{2}}}
\bm{L}_{\pi} & - (\dfrac{1-\sqrt{1-\lambda^{2}}}{1+\sqrt{1-\lambda^{2}}}\bm{I} \\
    \bm{I}  & \bm{0}
  \end{array} } \right].\]
  To any eigenvalue $\mu$ of $\bm{L}_{\pi}$ correspond (at most) two eigenvalues of $\bm{B}_{\pi}$, satisfying the following equation in $\omega$:
\begin{equation}\label{eq:sec-order-MVC-1}
  \omega^{2} - \dfrac{2 \lambda \mu}{1+\sqrt{1-\lambda^{2}}}\omega + \dfrac{1-\sqrt{1-\lambda^{2}}}{1+\sqrt{1-\lambda^{2}}}= 0.
  \end{equation}
  The discriminant of \eqref{eq:sec-order-MVC-1} is $\Delta = \dfrac{4 \lambda^{2} (\mu^{2}-1)}{(1+\sqrt{1-\lambda^{2}})^{2}}$. Note that $\mu \in \R$ because of our reversibility assumption. We therefore have the following expressions for the roots of \eqref{eq:sec-order-MVC-1}:
  \begin{align}
  \omega^{+}(\mu) & = \dfrac{\lambda}{1+\sqrt{1-\lambda^{2}}} \cdot \left(  \mu + i\sqrt{1-\mu^{2}}\right),\label{eq:w_plus_MVC_1}\\
    \omega^{-}(\mu) & = \dfrac{\lambda}{1+\sqrt{1-\lambda^{2}}} \cdot \left(  \mu - i\sqrt{1-\mu^{2}}\right). \label{eq:w_minus_MVC_1}
  \end{align}
  But $ | \omega^{+}(\mu)|  = | \omega^{-}(\mu)| = \dfrac{\lambda}{1+\sqrt{1-\lambda)^{2}}} \sqrt{ \mu^{2} + 1 - \mu^{2}} = \dfrac{\lambda}{1+\sqrt{1-\lambda)^{2}}}.$
%  
%  \begin{align*}
%  | \omega^{+}(\mu)| & = | \omega^{-}(\mu)| \\
%  & =  \dfrac{\lambda}{1+\sqrt{1-\lambda)^{2}}} |\mu - i\sqrt{1-\mu^{2}}|  \\
%  & = \dfrac{\lambda}{1+\sqrt{1-\lambda)^{2}}} \sqrt{ \mu^{2} + 1 - \mu^{2}} \\
%  & = \dfrac{\lambda}{1+\sqrt{1-\lambda)^{2}}}.
%  \end{align*}
  We can therefore conclude that
  
  \begin{align*}
  \rho(\bm{B}_{\pi}) & = \dfrac{\lambda}{1+\sqrt{1-\lambda)^{2}}} = \dfrac{\sqrt{1+\lambda}-\sqrt{1-\lambda}}{\sqrt{1+\lambda}+\sqrt{1\lambda}}= \dfrac{ 1-\sqrt{\kappa} }{1+\sqrt{\kappa}}.
\end{align*}  
\end{proof}
Theorem \ref{th:MVC-1} follows from combining Lemma \ref{lem:LTV-MVC} with Lemma \ref{lem:radius-tuning-MVC-1}.  {The fact that $0 < (1-\sqrt{1/\kappa})/(1+\sqrt{1/\kappa}) < \lambda$ for $\lambda \in (0,1)$ also follows directly from \eqref{eq:comparison-rate}.}

%and Lemma \ref{lem:rho-petit-o}:
%\begin{lemma}\label{lem:rho-petit-o}
%For any square matrix $\bm{B} \in \R^{n}$, for any initial vector $\bm{u}_{0} \in \R^{n}$, we have
%\[\bm{B}^{s}\bm{u}_{0} = o\left( \left( \rho(\bm{B}) + \eta \right)^{s} \right), \forall \; s \geq 0, \forall \; \eta >0.\]
%\end{lemma} 
%\proof{Proof of Lemma \ref{lem:rho-petit-o}.}
%Let $\bm{B} \in \R^{n \times n}, \bm{u}_{0} \in \R^{n}, \epsilon > 0.$
%We have 
%\begin{align*}
%\bm{B}^{s}\bm{u}_{0} & =\left(\rho(\bm{B}) + \epsilon\right)^{s} \left(\dfrac{\bm{B}}{\rho(\bm{B}) + \epsilon}\right)^{s}\bm{u}_{0},
%\end{align*}
%and \[ \rho \left( \dfrac{\bm{B}}{\rho(\bm{B}) + \epsilon} \right) = \dfrac{\rho(\bm{B})}{\rho(\bm{B} )+ \epsilon} <1.\]
%We can conclude using Proposition 7.10.5 in \citet{meyer-book}: \[ \forall \; \bm{B} \in \R^{n \times n}, \lim_{k \rightarrow + \infty} \bm{B}^{k} = \bm{0} \iff \rho(\bm{B}) < 1.\]
%\hfill \Halmos \endproof

\vspace{1mm}
\noindent {\textbf{Divergence of \ref{alg:MVC} on a non-reversible MDP instance.}}
We consider the MDP instance defined in Section \ref{sec:disc-reversibility}.
There are $n$ states, where $n$ is even and one action for all states. The transition matrix is defining a cycle. Therefore, the eigenvalues of the transition matrix are the $n-th$ roots of $1$ in $\C$. 
%For example, for $n=4$, the transition matrix $\bm{L}$ and its eigenspectrum are given by
% \begin{equation*}
%    \bm{L}=
%  \left[ {\begin{array}{cccc}
%   0 & 1 & 0 & 0 \\
%   0 & 0 & 1 & 0 \\
%   0 &0 & 0 & 1 \\
%   1 & 0 & 0 & 0 \\
%  \end{array} } \right], \; Sp(\bm{L}) = \{ 1,-1,i,-i\}.
%   \end{equation*}
% 
  Therefore, $i \in Sp(\bm{L})$. Let us compute $|\omega^{+}(i)|$, where $\omega^{+}$ is defined as \eqref{eq:w_plus_MVC_1}. For $\mu=i$, the discriminant becomes $\Delta = \dfrac{-8 \lambda^{2}}{(1+\sqrt{1-\lambda^{2}})^{2}}$, which leads to
   \begin{align*}
   |\omega^{+}(i)| & = | \dfrac{1}{2} \left( \dfrac{2 \lambda i}{1+\sqrt{1-\lambda^{2}}} + \dfrac{2 \lambda i \sqrt{2}}{1+\sqrt{1-\lambda^{2}}}  \right) |  =| \dfrac{\lambda}{1+\sqrt{1-\lambda^{2}}} \left( 1+\sqrt{2}\right)i | = \sqrt{3} \cdot \dfrac{\lambda}{1+\sqrt{1-\lambda^{2}}},
   \end{align*}
   and \ref{alg:MVC} may diverge for $\lambda$ close to $1$.
   
\section{Proof of Proposition \ref{prop:B-pi-IRU}.}\label{app:prop-B-pi-IRU} 
For $s \geq 0$ and $i \in \X$, we have $ T(\bm{h}_{s})_{i} = \max_{a \in \A} \{ r_{ia} + \lambda \cdot \bm{P}_{ia}^{\top}\bm{h}_{s} \},$ and therefore there exists a policy $\pi_{s}$ such that $T(\bm{h}_{s}) = T_{\pi_{s}}(\bm{h}_{s})=  \bm{r}_{\pi_{s}} + \lambda \cdot \bm{L}_{\pi_{s}} \bm{h}_{s}$.
Let $(\bm{v}_{s})_{s \geq 0}$ be the sequence of iterates of Algorithm \ref{alg:AVI} for the Bellman operator $T$. Similarly as for \eqref{eq:rec-x-AVC} in the proof of Lemma \ref{lem:LTV-AVC}, we can write
\begin{equation}\label{eq:rec-x-AVI}
\bm{x}_{s+1} = \bm{B}_{\pi_{s}}\bm{x}_{s} + \bm{b}_{\pi_{s}}.
\end{equation}
Let us write $\bm{x}^{*} = \left( \bm{v}^{*},\bm{v}^{*} \right),$ where $\bm{v}^{*}$ is the value function of the optimal policy.
We define $\bm{u}_{0} \in \R^{n}$ as $\bm{u}_{0} = \bm{x}_{0} - \bm{x}^{*}$. Now let us assume that  for a given $s \geq 0$, \begin{equation}\label{eq:hyp-rec}
   \bm{u}_{s}= \bm{x}_{s} - \bm{x}^{*}.
\end{equation}
Let us write $H_{\pi}: \R^{2 n } \rightarrow \R^{2 n}$ such that $H_{\pi}(\bm{x}) = \bm{B}_{\pi}\bm{x} + \bm{b}_{\pi}.$ By definition of the Bellman operator as the maximum of some affine operators, we have
$\bm{x}_{s+1} = \max_{\pi} H_{\pi}(\bm{x}_{s}), \forall \; s \geq 0,$
with the understanding that the maximum is taking row-by-row on the rows of the operator $\bm{B}_{\pi}$ and the vector $\bm{b}_{\pi}.$  
Therefore we have
\begin{align}
    \bm{x}_{s+1} & = \max_{\pi} H_{\pi}(\bm{x}_{s}) \nonumber \\
    & \geq H_{\pi^{*}} ( \bm{x}_{s}) \\
    & \geq H_{\pi^{*}}\left( \bm{x}^{*} \right) + \bm{B}_{\pi^{*}} \left(\bm{x}_{s} - \bm{x}^{*} \right) \nonumber \\
    & \geq \bm{x}^{*} + \bm{B}_{\pi^{*}} \left(\bm{x}_{s} - \bm{x}^{*} \right), \label{eq:step-1}
\end{align}
where \eqref{eq:step-1} follows from $H_{\pi^{*}}(\bm{x}^{*}) = \bm{x}^{*}.$
Now from \eqref{eq:hyp-rec}, $\bm{x}_{s} - \bm{x}^{*} = \bm{u}_{s},$ and therefore we have proved that
$  \bm{x}_{s+1} \geq \bm{x}^{*} + \bm{B}_{\pi^{*}}\bm{u}_{s}.$
Moreover,
\begin{align}
    \bm{x}_{s+1} & = \max_{\pi} H_{\pi}(\bm{x}_{s}) \nonumber \\
    & = H_{\pi_{s}} ( \bm{x}_{s}) \nonumber \\
    & = H_{\pi_{s}}\left( \bm{x}^{*} \right) + \bm{B}_{\pi_{s}} \left(\bm{x}_{s} - \bm{x}^{*} \right) \nonumber \\
    & \leq \bm{x}^{*} + \bm{B}_{\pi_{s}} \left( \bm{x}_{s} - \bm{x}^{*} \right),\label{eq:step-2}
\end{align}
where \eqref{eq:step-2} follows from $H_{\pi}(\bm{x}^{*}) \leq \bm{x}^{*}, \forall \; \pi \in \Pi,$ by definition of $\bm{x}^{*} =(\bm{v}^{*},\bm{v}^{*})$ and $\bm{v}^{*}$ being the fixed-point of the Bellman operator.
Again, using \eqref{eq:hyp-rec}, we can conclude that
$
    \bm{x}_{s+1} \leq \bm{x}^{*} + \bm{B}_{\pi_{s}} \bm{u}_{s}.
$
Let us define the next vector $\bm{u}_{s+1}$ as $\bm{u}_{s+1} = \bm{x}_{s+1} - \bm{x}^{*}.$
Overall, we have proved that
\begin{equation}\label{eq:recursion-u}
     \bm{x}_{s} = \bm{x}^{*} + \bm{u}_{s} \Rightarrow \bm{B}_{\pi^{*}} \bm{u}_{s} \leq \bm{u}_{s+1}  \leq  \bm{B}_{\pi_{s}} \bm{u}_{s}.
\end{equation}
We claim that \eqref{eq:recursion-u} implies that there exists a (randomized) policy $\hat{\pi}_{s}$ for which $\bm{u}_{s+1}=\bm{B}_{\hat{\pi}_{s+1}}\bm{u}_{s}.$ 
Let us write $\bm{u}_{s}=\left[ {\begin{array}{c}
    \bm{t}_{s} \\
   \bm{d}_{s} \\
  \end{array} } \right]$, for some vectors $\bm{t}_{s},\bm{d}_{s} \in \R^{n}.$ For the sake of brevity, let us write $\theta = 1 - \sqrt{1/\kappa}.$
  Then \eqref{eq:recursion-u} can be written 
  \begin{align*}
  \left[ {\begin{array}{c}
    \left( \bm{I} + \bm{L}_{\pi^{*}} \right) \left( \theta \bm{t}_{s} - (\theta^{2}/2)\bm{d}_{s} \right)    \\
   \bm{t}_{s} \\
  \end{array} } \right] & \leq \left[ {\begin{array}{c}
    \bm{t}_{s+1} \\
   \bm{d}_{s+1} \\
  \end{array} } \right], \\
  \left[ {\begin{array}{c}
    \bm{t}_{s+1} \\
   \bm{d}_{s+1} \\
  \end{array} } \right] &   \leq \left[ {\begin{array}{c}
     \left( \bm{I} + \bm{L}_{\pi_{s}} \right) \left( \theta \bm{t}_{s} - (\theta^{2}/2)\bm{d}_{s} \right) \\
   \bm{t}_{s} \\
  \end{array} } \right]
  \end{align*}
%\[
%\left[ {\begin{array}{c}
%    \left( \bm{I} + \bm{L}_{\pi^{*}} \right) \left( \theta \bm{t}_{s} - (\theta^{2}/2)\bm{d}_{s} \right)    \\
%   \bm{t}_{s} \\
%  \end{array} } \right] \leq \left[ {\begin{array}{c}
%    \bm{t}_{s+1} \\
%   \bm{d}_{s+1} \\
%  \end{array} } \right],
%\]
%\[
%\left[ {\begin{array}{c}
%    \bm{t}_{s+1} \\
%   \bm{d}_{s+1} \\
%  \end{array} } \right]    \leq \left[ {\begin{array}{c}
%     \left( \bm{I} + \bm{L}_{\pi_{s}} \right) \left( \theta \bm{t}_{s} - (\theta^{2}/2)\bm{d}_{s} \right) \\
%   \bm{t}_{s} \\
%  \end{array} } \right].
%\]
  This readily implies that $\bm{d}_{s+1}=\bm{t}_{s}.$ Moreover, for each row $i \in \{1,...,n\}$, we have 
  \[t_{s+1,i} \in \left[ \left(\left( \bm{I} + \bm{L}_{\pi^{*}} \right) \left( \theta \bm{t}_{s} - (\theta^{2}/2)\bm{d}_{s} \right) \right)_{i},\left( \left( \bm{I} + \bm{L}_{\pi_{s}} \right) \left( \theta \bm{t}_{s} - (\theta^{2}/2)\bm{d}_{s} \right) \right)_{i} \right],\] which for some appropriate coefficients $\mu_{i} \in [0,1], 1 \leq i \leq n$, can be written, 
  \[t_{s+1,i} = \mu_{i} \left(\left( \bm{I} + \bm{L}_{\pi^{*}} \right) \left( \theta \bm{t}_{s} - (\theta^{2}/2)\bm{d}_{s} \right) \right)_{i}+ (1-\mu_{i})\left(\left( \bm{I} + \bm{L}_{\pi_{s}} \right) \left( \theta \bm{t}_{s} - (\theta^{2}/2)\bm{d}_{s} \right)  \right)_{i}.\]
   Therefore we construct the policy $\hat{\pi}_{s}$ such that $\hat{\pi}_{s, ia} = \mu_{i} \pi^{*}_{ia} + (1-\mu_{i}) \pi_{s,ia}, \forall \; (i,a) \in \X \times \A,$ and we obtain $\bm{t}_{s+1} =\left( \bm{I} + \bm{L}_{\hat{\pi}_{s}} \right) \left( \theta \bm{t}_{s} - (\theta^{2}/2)\bm{d}_{s} \right),$ i.e., we have proved that
  $\bm{u}_{s+1} = \bm{B}_{\hat{\pi}_{s}} \bm{u}_{s}.$
  We can conclude that for any $s \geq 0$,
  $\bm{x}_{s} = \bm{x}^{*} +  \bm{u}_{s} \Rightarrow \exists \; \hat{\pi}_{s}, \bm{x}_{s+1} = \bm{x}^{*} +  \bm{B}_{\hat{\pi}_{s}}\bm{u}_{s}.$
  Since this is also true for $s=0$, we can conclude that there exists a sequence of policies $(\hat{\pi}_{s})_{s \geq 0}$ such that
$\bm{x}_{s} = \bm{x}^{*} +  \bm{B}_{\hat{\pi}_{s-1}} \cdot ... \cdot \bm{B}_{\hat{\pi}_{0}}(\bm{x}_{0}-\bm{x}^{*}), \forall \; s \geq 1.$

\section{Details on the numerical experiments}\label{app:details-numerical-exp}
\subsection{Algorithms}\label{app:details-other-algorithms}
{
In this appendix we present some details about the algorithms used in our numerical experiment section. 
Value Iteration and Gauss-Seidel Value Iteration are first-order methods for MDPs. Anderson Value Iteration is related to quasi-Newton methods for MDPs, while Policy Iteration is equivalent to Newton methods. We refer the reader to \cite{grand2021convex} for a classification of algorithms as first-order, second-order and quasi-Newton methods for MDPs.
\paragraph{Gauss-Seidel Value Iteration.}  Gauss-Seidel Value Iteration (GS-VI) is an asynchronous version of Value Iteration, and its theoretical convergence rate is at least as fast as Value Iteration~\citep{Puterman}:
\[
v_{s+1,i} = \max_{a \in \A} r_{sa} + \lambda \cdot \sum_{j=1}^{i-1} P_{iaj'}v_{s+1,j} + \lambda \cdot \sum_{j=i}^{n} P_{iaj}v_{s,j}.\]
\paragraph{Policy Iteration.} Policy Iteration (PI)~\citep{ye-2011} alternates between \textit{value computation} steps \eqref{eq:PI-value-comp} and \textit{policy improvement} steps \eqref{eq:PI-policy-improv}: it starts at $\bm{v}_{0} \in \R^{n}$ and returns a sequence of policies $\left( \pi_{s} \right)_{s \geq 0}$ such that for all $s \geq 0$,
\begin{align}
T(\bm{v}_{s}) & = T_{\pi_{s}}(\bm{v}_{s}),  \label{eq:PI-policy-improv} \\
\bm{v}_{s+1} & = \bm{v}^{\pi_{s}}. 
\label{eq:PI-value-comp}
\end{align}
\paragraph{Anderson Value Iteration.} This algorithm combines the last $m+1$-iterates $T(\bm{v}_{s}), ..., T(\bm{v}_{s-m})$ to obtain the next iterate $\bm{v}_{s+1}$:
\[ \bm{v}_{s+1} = \sum_{s'=0}^{m} \alpha_{s'} T(\bm{v}_{s-m+s'}).\]
The weights $\alpha_{0}, ..., \alpha_{m}$ are updated at every step $s$. We refer to Algorithm 1 and Equation (1) in \cite{ref-c} for further details. There is no heuristics for choosing $m$; we choose $m=5$ in our numerical experiments and initialize $\bm{v}_{1} = T(\bm{v}_{0}), ..., \bm{v}_{m} = T(\bm{v}_{m-1})$.}
{
\subsection{MDP instances}\label{app:details-MDP-instances}
We now provide some details about the MDP instances used in our simulations of Section \ref{sec:simu}.
\subsubsection{Garnet MDPs.}
 Garnet MDPs were introduced in \cite{garnet} to provide a class of MDP parametrized by a branching factor $n_{\sf branch}$, equal to the proportion of reachable next states from each state-action pair $(i,a)$.  For each state $i$ and action $a$,  a number $n_{\sf next}=\lfloor n_{\sf branch} \times n \rfloor$ of possible next states are chosen uniformly at random from the set of states. Then $n_{\sf next}-1$ scalars at chosen at uniformly random in $[0,1]$, and the lengths of the corresponding intervals give the likelihood of transitioning toward each next state. We use $n_{branch} = 80 \%,A=50$ in our numerical experiments and draw the reward parameters at random uniformly in $[0,100]$.
 }
 \subsubsection{Forest MDP.}
 {
 In this instance, the state represents the current age of the forest; the decision maker earns reward for cutting and selling wood.   Every year, wildfire may burn and bring the forest back to the youngest age.  A complete description of this model can be found in \cite{pymdp}, where the first mention of this model is from \cite{possingham1997application}.  

\paragraph{States.} There are $n$ states. The state $1$ is the youngest state for the forest. The forest can not grow beyond state $n$. 
\paragraph{Actions.} The two actions are \textit{wait} and \textit{cut \& sell}.
\paragraph{Transitions. } When the action is \textit{wait}, the forest may transition to state $i+1$ from state $i$ (the forest grows with probability $1-p$) or to state $1$ with probability $p$ (the forest burns). When the action is \textit{cut \& sell}, the next state is always state $1$. We use $p=0.05$ for the probability of wildfire.
\paragraph{Rewards.} 
There is a reward of $4$ when the forest reaches the oldest state (state $n$) and the chosen action is \textit{wait}.  There is a reward of $0$ at every other state if the chosen action is \textit{wait}. When the action is \textit{cut \& sell}, the reward at the youngest state $i=1$ is $0$, there is a reward of $1$ in any other state $i \in \{1, ..., n-1\}$, and a reward of $2$ in $i=n$. 
\subsubsection{Healthcare MDP.}
In this instance,  the state represents the health condition of the patient. The decision maker has to choose different levels of invasiveness for the treatment (\textit{low}, \textit{medium} or \textit{high} level of drug). The goal is to minimize the invasiveness of the treatment while avoiding the \textit{mortality} terminal state. This is inspired from the MDPs presented in \cite{grand2020robust} for modeling proactive transfers of patients to the Intensive Care Unit and \cite{Goh} for  fecal immunochemical testing. A representation of the MDP with $5$ states and the mortality state can be found in Figures \ref{fig:MDP-1}-\ref{fig:MDP-3}.
\paragraph{States.}  There are $n$ states.  State $1$ (resp.  state $n-1$) represents the healthiest (resp. the worst) health condition of the patient. State $n$ is the mortality state (labeled $m$ here).
\paragraph{Actions.} There are three actions, \textit{low}, \textit{medium} and \textit{high}, corresponding to different levels of invasiveness of the treatment. 
\paragraph{Transitions.} The mortality state is an absorbing state. More invasiveness treatments are more effective at bringing the next state toward the healthier states but also more expensive. The details of the transitions are given in Figures \ref{fig:MDP-1}-\ref{fig:MDP-3}.

\begin{figure*}[htp]
  \centering
  \begin{subfigure}{0.3\textwidth}\centering
  \includegraphics[width=0.8\linewidth]{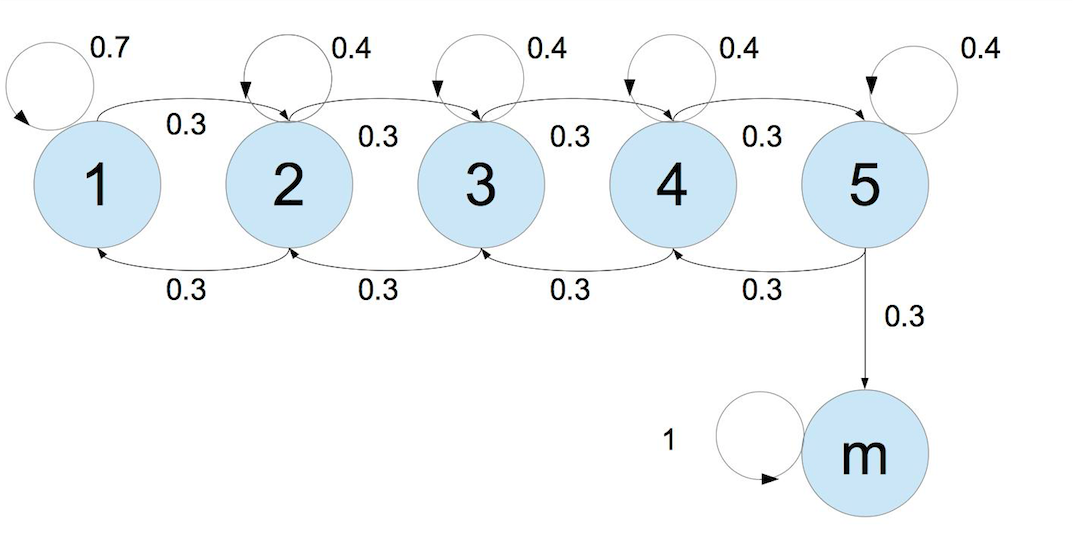}
  \captionof{figure}{ Transition for action = \textit{low drug level}.}
  \label{fig:MDP-1}
\end{subfigure}
  \begin{subfigure}{0.3\textwidth}\centering
\centering
  \includegraphics[width=0.8\linewidth]{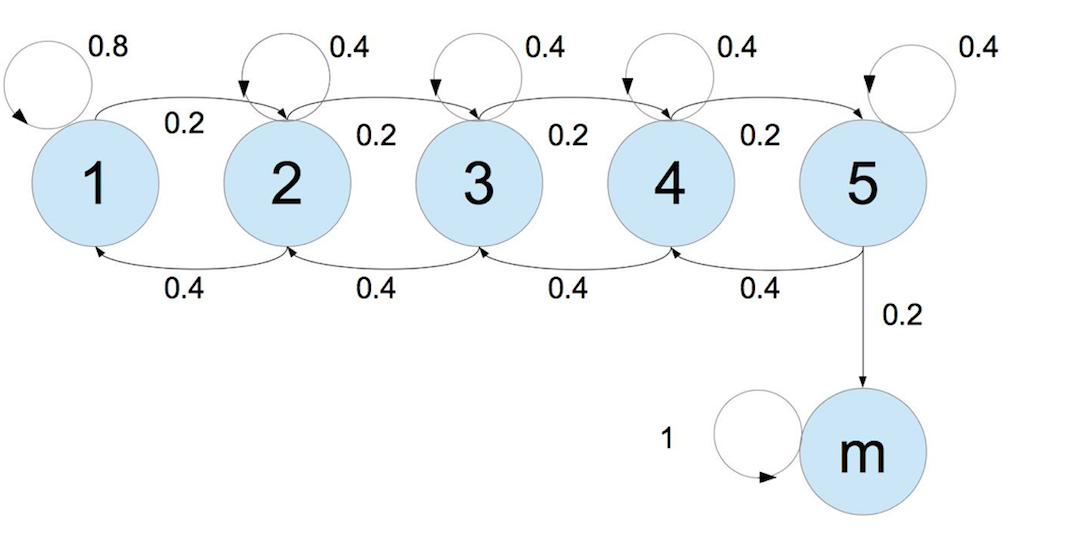}
  \captionof{figure}{ Transition for action = \textit{medium drug level}.}
  \label{fig:MDP-2}
\end{subfigure}
  \begin{subfigure}{0.3\textwidth}\centering
\centering
  \includegraphics[width=0.8\linewidth]{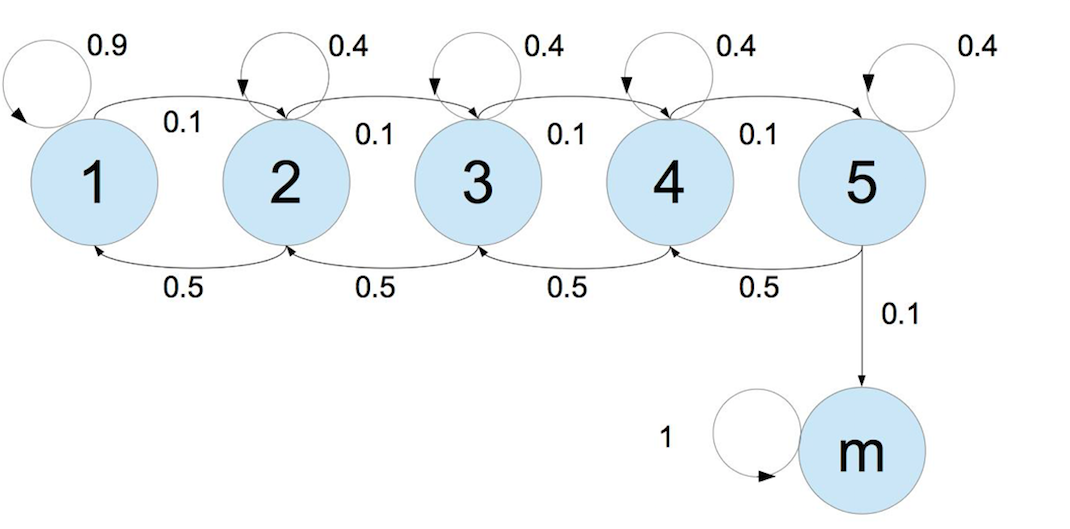}
  \captionof{figure}{ Transition for action = \textit{high drug level}.}
  \label{fig:MDP-3}
\end{subfigure}
\caption{Transition rates for the healthcare MDP instance.}
\end{figure*}
}
%%%%%%%%%%%%%%%%%
\end{document}